\documentclass[11pt]{amsart}

\PassOptionsToPackage{linktocpage}{hyperref}

\usepackage[french, english]{babel}
\usepackage{a4wide}
\usepackage[utf8,latin1]{inputenc}
\usepackage{leftidx}%left index

\usepackage{amsmath,amscd,amssymb,amsthm,epsfig,array,hhline,longtable,float,mathtools}
\restylefloat{table}
\usepackage{chngcntr}
\counterwithin{table}{section}
\usepackage{bbm} 
\usepackage{url}

\usepackage{hyperref} 

\usepackage[all]{xy}
\usepackage[T1]{fontenc}

\usepackage[all]{xy}
\usepackage{hyperref}
\usepackage{array}
\usepackage{tabularx}
\usepackage{yfonts}
\usepackage{ulem}
\usepackage{setspace} 
\usepackage{abstract}

\usepackage{geometry}
\usepackage{subfig}

\makeatletter
\renewcommand\subsection{\@startsection{subsection}{3}{\z@}%
                                     {-3.25ex\@plus -1ex \@minus -.2ex}%
                                     {-1em}% <---changed
                                     {\normalfont\normalsize\bfseries}}
\makeatother

\newtheorem{theorem}{Theorem}[section]
\newtheorem{lemm}[theorem]{Lemma}
\newtheorem{definition}[theorem]{Definition}
\newtheorem{remark}[theorem]{Remark}
\newtheorem{example}[theorem]{Example}
\newtheorem{coro}[theorem]{Corollary}
\newtheorem{prop}[theorem]{Proposition}

\newtheorem{prop-def}[theorem]{Proposition-Definition}

\numberwithin{equation}{section}

\newcommand{\Hom}{\mathrm{Hom}}

\renewcommand{\bar}{\overline}
\renewcommand{\d}{\mathrm{d}}

\renewcommand{\hat}{\widehat}
\newcommand{\vol}{\mathrm{vol}}

\newcommand{\Gal}{\mathrm{Gal}}

\newcommand{\ad}{\mathrm{ad}}
\newcommand{\Ad}{\mathrm{Ad}}

\newcommand{\pr}{\mathrm{p}}

\DeclareMathOperator{\Spec}{Spec}

\newcommand{\AAA}{\mathbb{A}}

\renewcommand{\ggg}{\mathfrak{g}}

\newcommand{\kkk}{\mathfrak{k}}
\renewcommand{\lll}{\mathfrak{l}}
\newcommand{\nnn}{\mathfrak{n}}
\newcommand{\mmm}{\mathfrak{m}}
\newcommand{\ppp}{\mathfrak{p}}

\newcommand{\zzz}{\mathfrak{z}}
\newcommand{\ttt}{\mathfrak{t}} 
 
\newcommand{\ooo}{\mathcal{O}}

\newcommand{\ago}{\mathfrak{a}}

\renewcommand{\leq}{\leqslant}
\renewcommand{\geq}{\geqslant}
\newcommand{\htau}{\hat \tau}

\makeatletter
\newcommand*{\rom}[1]{\expandafter\@slowromancap\romannumeral #1@}
\makeatother

\usepackage[pagewise]{lineno}%\linenumbers

\begin{document}
	%rajouter les numérotation pour les \subsubsectione et \subparagraphe

%\selectlanguage{french}

\setcounter{secnumdepth}{3}
\setcounter{tocdepth}{1}

\hypersetup{							% Information sur le document
pdfauthor = {HONGJIE YU},			% Auteurs
pdftitle = {a coarse geometric expansion},			% Titre du document
%pdfkeywords = {Tag1, Tag2, Tag3,...},	% Mots-clefs
%pdfstartview={FitH}
}					% ajuste la page à la largueur de l'écran
\title[A coarse geometric expansion and some applications]{A coarse geometric expansion of a variant of Arthur's truncated traces and some applications
} 
\author{HONGJIE YU }
\address{IST Austria, Am Campus 1, 3400 Klosterneuburg, Austria. Current address: Weizmann Institute of Science, Herzl St 234, Rehovot, Israel \\Email: \url{hongjie.yu@weizmann.ac.il}}
%\date{\vspace{-2ex}}
\maketitle

\renewcommand{\abstractname}{Abstract}
\begin{abstract}
Let $F$ be a global function field with constant field $\mathbb{F}_q$. Let $G$ be a reductive group over $\mathbb{F}_q$. 
We establish a variant of Arthur's truncated kernel for $G$ and for its Lie algebra which generalizes Arthur's original construction. We establish a coarse geometric expansion for our variant truncation. 

As applications, we consider some existence and uniqueness problems of some cuspidal automorphic representations for the functions field of the projective line $\mathbb{P}^1_{\mathbb{F}_q}$ with two points of ramifications. 
\end{abstract}

\tableofcontents

\section{Introduction}
In \cite{A1}, Arthur defines a truncated kernel, and establishes some of its basic properties. He then gives a coarse geometric expansion whose integrals give the geometric side of his non-invariant trace formula. Arthur's work is over a number field.  %Much of the results of \cite{A1} hold for a global function field too, although not obviously, if the characteristic of the field is large enough with respect to the group (in another point of view, if the group is small with respect to the characteristic). %so that there exist a logarithm map on the unipotent radical and Jordan decompositions exists. 
Later in \cite{Laf}, Lafforgue shows that Arthur's truncated kernel is well behaved over a function field for $GL_n$ as well as its inner forms. There is no restriction on the characteristic, while there is no coarse geometric expansion either, because of the lack of the Jordan-Chevalley decomposition. Lafforgue used the Weil's dictionary between adeles and vector bundles. He shows that the truncated kernel over a function field has some nicer properties than that over a number field.

Let $F$ be a global function field with constant finite subfield $\mathbb{F}_q$. In this article, we generalize Lafforgue's results to a reductive group over $\mathbb{F}_q$. We also consider a variant version of Arthur's truncated kernel inspired by the $\xi$-stability of Hitchin pairs defined and studied by Chaudouard and Laumon (\cite{C-L1}). For some special test functions, this variant version will have nicer properties than the Arthur's original version. 
We also establish a Lie algebra version of non-invariant trace formula following the number field case given by Chaudouard (\cite{ChauLie}).

We explain some phenomena that do not show up for a number field first.

If $p\mid |\pi_1(G^{der})|$, there exists a unipotent element which is not contained in any proper parabolic subgroup, for example the element  \[ u=\begin{pmatrix}
0 & 1\\ c &0 
\end{pmatrix}
 \]
in $PGL_2(F)$ if $char(F)=2$ and $c$ is not a square element. Fortunately, Gille \cite[Theorem 2]{Gille} (see Proposition \ref{Gille}) has shown that such phenomenon does not happen when $p\nmid |\pi_1(G^{der})|$ for a field such that $[F:F^p]\leq p$. We'll extend Arthur's constructions to function field under this hypothesis. Note that if $G^{der}$ is simply connected or the characteristic is very good for $G$ (in the sense of \cite[I.4]{SS}), then the hypothesis is satisfied. 

Another problem is the lack of Jordan-Chevalley decomposition. Recall that we say that an element $\gamma\in G(F)$ admits a Jordan decomposition if its unipotent part and semi-simple part of the Jordan decomposition in $G(\bar{F})$ lies in $G(F)$. To have a geometric expansion of the truncated trace,  we need to define a partition on $G(F)$:
\begin{equation} \label{partition} G(F)=\coprod_{o\in \mathcal{E}}o, \end{equation}
which should be compatible with conjugations.  
Over a number field, Arthur defines an equivalence relation on $G(F)$ so that two elements are equivalent if their semi-simple parts are conjugate. 
Such definitions do not work directly for a field in positive characteristic, but we can define the same equivalence relation only on those elements admitting a Jordan-Chevalley decomposition and putting those which do not admit Jordan-Chevalley decomposition in a single class. Then we give a coarse expansion under this partition. 

To be precise, with notations commonly used in the literature which will be introduced in the Section \ref{notation}, fixing two truncation parameters $X\in \ago_B$ and $\xi\in \ago_{B,\infty}$ (where Arthur uses $T$ for $X$ and $\xi$ is an additional parameter depending on a fixed place $\infty$), we construct following Arthur a linear function  \[ J^{G,\xi,X}: \mathcal{C}_c^{\infty}(G(\AAA))\rightarrow \mathbb{C}, \]
which generalizes Arthur's ``truncated trace''. For each $o\in\mathcal{E}$ as in \eqref{partition}, we also define a linear function  \[ J_o^{G,\xi,X}: \mathcal{C}_c^{\infty}(G(\AAA))\rightarrow \mathbb{C}, \]
such that \[ J^{G,\xi,X}=\sum_{o\in\mathcal{E}} J_{o}^{G,\xi,X}. \]
Note that we always assume $p\nmid |\pi_1(G^{der})|$. 
For $GL_n$, Lafforgue has given a very simple formula for Arthur's truncated trace when the truncation parameter is taken to be deep enough in the positive chamber, we show that there is a similar formula for $J_{o}^{G,\xi,X}$ for any constant reductive group (cf. Theorem \ref{Main}). In addition to the group version, we also establish a coarse trace formula for the Lie algebra of $G$ under more restrictive hypothesis on the characteristics. 

Finally, let me explain about the reason  for introducing another parameter $\xi$ and a version for a Lie algebra. 

Over a number field, the coarse geometric expansion is refined by Arthur into a sum of weighted orbital integrals, where the weight factor is given by a volume of convex envelope in a vector space. Over a function field, such volume will be replaced by a count of a lattice points bounded by a convex envelope, which is not so pleasant because of those lattice points lying on the boundary of the convex envelope. The additional parameter $\xi$ is aimed to translate the convex envelope so that there are no longer points on the boundary if $\xi$ is in general position. In fact,  Chaudouard  and Laumon (\cite[11.14.3]{C-L1}) showed that for regular semi-simple weighted orbital integrals, by introducing a $\xi$ in general position, one recovers Arthur's weight factor. 

Moreover, we showed in \cite[Appendix B]{Yu} that $J^{\ggg, \xi}(f)$ (where $X$ is taken to be $0$) for some simple test function $f\in \mathcal{C}_c^{\infty}(\ggg(\AAA))$ gives the groupoid volume of semi-stable parabolic Hitchin bundles with a stability parameter corresponding to $\xi\in \ago_B$.
Therefore a Lie algebra version could be useful to the study of Hitchin's moduli spaces. 
For the group $GL_n$, it has been studied by Chaudouard in \cite{Chau}. We study it for a more general group in \cite{Yu}.% Indeed, in \cite{Yu}, we will take advantage of a $\xi$ in general position and link the geometric side of the trace formula for some Lie algebras and for a group by the coarse expansion established in this article applied to a small class of test functions. This method is also used for applications in the section \ref{}

\begin{remark}
This work is motivated by an application in \cite{Yu}. 
In the preparation of these articles, there appears the work of Labesse and Lemaire \cite{LL} that, among many other results, also establishes a coarse geometric expansion of Arthur's non-invariant trace formula. They give a coarse geometric expansion by introducing the notion of primitive pairs %and associate each $\gamma\in G(F)$ a unique conjugacy class of primitive pair, which induce a partition for $G(F)$ 
(cf. \cite[3.3]{LL}) which works in any characteristic. 
Note that under the hypothesis $p\nmid |\pi_1(G^{der})|$, the coarse expansion given by them coincides with the expansion in this article for those classes which admit a Jordan decomposition (which can be seen from the beginning of the proof of the Proposition \ref{unipotent}), and is  finer otherwise. 
After these remarks, theirs works overlap the the main results of the current paper and are more general than the cases treated here.

There still remain some novelties for the constructions in this article. First of all, the approach here is different from \cite{LL}. The key ingredient in our proof is that we use Behrend's complementary polyhedra and semi-stable reductions of $G$-bundles as the source of reduction theory and some combinatoric lemmas. A direct advantage is that we do not need the truncation parameter $X$ to be ``regular enough'' to have integrability of the truncated kernel. 
Secondly, for $GL_n$, Lafforgue has given a very simple formula for Arthur's truncation when the truncation parameter is taken to be deep enough in the positive chamber, we generalize it to any constant reductive group (cf. Theorem \ref{Main}). Lastly, the introduction of an additional parameter $\xi$ and the trace formula for Lie algebras are new for a function field and they're not useless as showed by the applications in the Section \ref{appl} and in \cite{Yu}.
\end{remark}

\subsection{Some applications}
The main theorems of this article are stated in the Section \ref{maintrace}. To motivate the reader, we present here some applications that partially confirm a conjecture made to me personally by Harris. Note that the requirement that $\xi$ is chosen in general position and our Lie algebra version trace formula are essentially used in the proofs. 

We consider the case that $F=\mathbb{F}_q(t)$ is the function field of the projective line $\mathbb{P}^1=\mathbb{P}^{1}_{\mathbb{F}_q}$. Let $G$ be a split reductive group defined over $\mathbb{F}_q$ such that the characteristic $p$ is very good for $G$ (cf. \ref{chara}). Let $\lambda, \mu \in \mathbb{P}^1(\mathbb{F}_q)$ be two distinct $\mathbb{F}_q$-rational points, identified as two places of $F$. 
Let $T_\lambda$ and $T_\mu$ be two maximal torus of $G$ defined over $\mathbb{F}_q$. Suppose that $\theta_\lambda$ and $\theta_\mu$ are characters in general position of $T_\lambda(\mathbb{F}_q)$ and $T_\mu(\mathbb{F}_q)$ respectively. We obtain by Deligne-Lusztig induction (fixing any isomorphism between $\mathbb{C}$ and $\overline{\mathbb{Q}}_\ell$) two irreducible representations 
\[\rho_\lambda=\epsilon_{T_\lambda}\epsilon_G R_{T_\lambda}^{G}(\theta_\lambda) \text{ and }  \rho_\mu=\epsilon_{T_\mu}\epsilon_G R_{T_\mu}^{G}(\theta_\mu)\] of $G(\mathbb{F}_q)$, where $\epsilon_{H}=(-1)^{rk_{\mathbb{F}_q}H}$ is the sign according to the parity of the split rank of $H$. 
Let $\mathcal{O}_\lambda$ and $\mathcal{O}_\mu$ be the local ring in $\lambda$ and $\mu$ respectively. 
By inflation, we get irreducible representations of $G(\mathcal{O}_\lambda)$ and $G(\mathcal{O}_\mu)$ respectively, still denoted by $\rho_\lambda$ and $\rho_\mu$.

\begin{theorem}[Theorem \ref{nonexistence}]
Suppose that $T_\lambda$ is split and $G$ is not a torus. Let $M_\mu$ be the minimal Levi subgroup containing $T_\mu$ defined over $\mathbb{F}_q$. Without loss of generality, we suppose that $M_\mu$ contains $T_\lambda$. Under the hypothesis that
 \begin{equation}\label{1.1} \theta_\lambda^w \theta_\mu |_{Z_M(\mathbb{F}_q)} \neq 1,  \end{equation}
for any Weyl element $w\in W$ and for any maximal proper semi-standard Levi subgroup $M$ of $G$ that contains $M_\mu$. 
There is no cuspidal automorphic representation $\pi=\otimes'\pi_v$ of $G(\AAA)$ such that $\pi_\lambda$ contains $(G(\mathcal{O}_\lambda), \rho_\lambda)$,  $\pi_\mu$ contains $(G(\mathcal{O}_\mu), \rho_\mu)$ and all other local components are unramified. 
\end{theorem}
\begin{remark}
If the torus $T_\mu$ is maximally anisotropic over $\mathbb{F}_q$, then the hypothesis \eqref{1.1}  holds automatically. It is expected (communicated to me by Harris) that the theorem holds under the weaker hypothesis that $\rho_\lambda$ is not contragredient to $\rho_\mu$ and we should have multiplicity one  when $\rho_\lambda$ is contragredient to $\rho_\mu$ at least when $G$ is split semi-simple and simply connected. We verify this statement in a special case in the Theorem \ref{SLl}.  Note that we have showed in \cite{Yu} that if we allow at least three points of ramification, then there are ``many'' such automorphic cuspidal representations when $q$ is large. 
\end{remark}

\begin{theorem}[Theorem \ref{102}]\label{SLl}
Let $G=SL_l$ with $l$ being a prime number different from $p$. Suppose that $\rho_\lambda$ and $\rho_\mu$ constructed before are cuspidal. 
 If $\rho_\lambda$ is  contragredient to $\rho_\mu$, then there is exactly one cuspidal automorphic representation $\pi=\otimes'\pi_v$ of $G(\AAA)$ (with cuspidal multiplicity $1$) such that $\pi_\lambda$ contains $(G(\mathcal{O}_\lambda), \rho_\lambda)$,  $\pi_\mu$ contains $(G(\mathcal{O}_\mu), \rho_\mu)$ and all other local components are unramified. 
If $\rho_\lambda$ is not contragredient to $\rho_\mu$, then there is no such cuspidal automorphic representation. 
\end{theorem}

\subsection{Structure of the article}
In Section \ref{alggr}, we discuss the restrictions on the characteristic. 
The main theorems are announced  in Section \ref{maintrace}. To state the constructions (mainly to introduce our variant parameter $\xi$), we need preparations in Section \ref{u}. Later sections are devoted to the proof of theorems in the Section \ref{maintrace}. It is in Section \ref{proofm} where the main theorems of this article are proved. 
In Section \ref{essunip}, we present another expression for the truncated trace which is relatively easier to use. In Section \ref{appl}, we give some applications.

\subsection{Acknowledgement}
I'd like to thank Prof. Chaudouard for introducing me to this area. I'd like to thank Prof. Harris for asking me the question that makes  Section \ref{appl} possible. 
I'm grateful for the support of Prof. Hausel and IST Austria. The author was funded by an ISTplus fellowship: This project has received funding from the European Union's Horizon 2020 research and innovation programme under the Marie Sk\l odowska-Curie Grant Agreement No. 754411.

\section{Notations}\label{notation}
\subsection{Notations for the groups}
We use the language of group schemes, so a linear algebraic group $G$ over a ring $k$ is an affine group scheme $G\rightarrow \Spec(k)$ of finite type. We do not require it to be be smooth nor connected. Reductive groups are connected smooth group schemes with trivial geometric unipotent radical. Intersection and center are taken in the scheme theoretic sense.

%So let $G$ act on a variety $V$ defined over $k$, for any $v\in V(k)$, the stabilizer of $v$ in $G$ is automatically defined over $k$, but it will not be. 
%We will only need to treat the case that $R$ is a field or $R$ is a local ring. The kernel of a morphism $G\rightarrow H$ is defined as the scheme that represents the functor $R\mapsto ker(G(R)\rightarrow H(R))$.  

$F, \bar{F}, \mathbb{F}_q$:  Throughout the article $F$ is a global function field with finite constant field $\mathbb{F}_q$, i.e. a field of rational functions over a smooth, projective and geometrically connected curve $C$ defined over $\mathbb{F}_q$. We fix an algebraic closure $\bar{F}$ of $F$.
%The same manner of notations applies for other fields appearing in this article, for example $\bar{\mathbb{F}}_q=\mathbb{F}_q^s$ is a fixed algebraic closure of $\mathbb{F}_q$. 

$G^0$, $Z_G$, $G^{der}$, $G^{sc}$:   For a linear algebraic group $G$, we denote by $G^0$, $Z_G$, $G^{der}$ and $G^{sc}$ the connected component of the identity of $G$, the center of $G$, and the derived group of $G$, and the simply connected covering of $G^{der}$ respectively.

 $\ggg$: We denote by $\ggg$ the Lie algebra of $G$. In general, if a group is denoted by a capital letter, we use its lowercase gothic letter for its Lie algebra.

$G_x$, $G_X$, $C_G(H)$, $N_G(H)$:   When $G$ is defined over a field $F$ and $x\in G(F)$, and $X\in \ggg(F)$, we denote by $G_{x}$ (resp. $G_X$) the centralizer of $x$ (resp. the adjoint centralizer of $X$) in G. If $H$ is a subgroup of $G$ defined over $F$, we denote by $C_G(H)$ and $N_G(H)$ the centralizer and normalizer of $H$ in $G$ respectively. 
%From now on, $G$ will always be a reductive group defined over a finite field $\mathbb{F}_q$. 

$G$, $B$, $T$, $A$, $\Xi_G$: Suppose that $G$ is a connected reductive group defined over $\mathbb{F}_q$. Let's fix a Borel subgroup $B$ defined over $\mathbb{F}_q$, a maximal torus $T$ defined over $\mathbb{F}_q$ contained in $B$. Let $A$ be the maximal subtorus of $T$ that is defined and split over $\mathbb{F}_q$. We fix a lattice $\Xi_G$ in $Z_G(F)\backslash Z_G(\AAA)$, i.e. a finitely generated free abelian group so that $Z_G(F)\backslash Z_G(\AAA)/\Xi_G$ is compact. For a field $k$ containing $\mathbb{F}_q$, we denote by $G_k$, or simply by $G$ itself if there is no harm, the base change of $G$ to $k$.

\subsection{Arthur's notations}

Let $|F|$ be the set of places of $F$. For each $v\in |F|$, let $F_v$ be the completion of $F$ at $v$, $\ooo_v$ be the  ring of integers of $F_v$ and $\kappa_v$ be the residue field of $\ooo_v$. Let $\AAA$ be the ring of adèles of $F$, and $\ooo$ be the ring of integral adèles.

We denote by $W^{(G,T)}=N_G(T)/C_G(T)$ be the Weyl group of $G$ with respect to $T$, understood as a finite étale group scheme over $F$ and \[W= W^{(G,T)}(F)=W^{(G,T)}(\mathbb{F}_q). \]
For each $s\in W$ we fix a representative $w_s\in N_G(T)(\mathbb{F}_q)$.

%Note that any standard parabolic subgroup defined over $F$ has a $\mathbb{F}_q$ structure and $W^{(G,T)}(F)=W^{(G,T)}(\mathbb{F}_q)$ since the Weyl group $W^{(G,T)}$ is a finite étale. 

%For any linear algebraic group $P$ defined over $F$, let $X^*(P):=\mathrm{Hom}_F(P, \mathbb{G}_{m})$ and $\ago_P:=X^{*}(P)\otimes\mathbb{R}$. Let $A_P$ be the maximal split torus defined over $F$ in the center of $P$. Then we know that $X^*(P)\subseteq X^*(A_P)$ which is cofinite, hence $X^{*}(A_P)$ is also a lattice in $\ago_P$. 

We will  use notations introduced by Arthur which are now more or less standard in this area. An excellent source to notations and their properties is \cite[Chapitre I]{LabWal}. For reader's convenience, instead of sending the reader to various articles of Arthur, we will prefer to cite this book if possible.

Throughout this article, we reserve letters $P, Q$ for parabolic subgroups, $M$ for a Levi subgroup of $G$ (i.e. Levi subgroup of a parabolic subgroup of $G$).  When $P,Q$ are semi-standard, we denote by $M_P$  the semi-standard Levi subgroup of $P$ and $N_P$ for the unipotent radical of $P$. 

\begin{enumerate}
%\item [$\bullet$] 
\item[$\bullet$] $\mathcal{P}^{Q}(M)$ parabolic subgroups contained in $Q$ which admit $M$ as Levi subgroup; 
\item[$\bullet$] $\mathcal{P}(B)$ standard parabolic subgroups of $G$;%, i.e. parabolic subgroups defined over $F$ containing $B$; 
\item[$\bullet$] $X^{*}(P)_F=\Hom_F(P,\mathbb{G}_{m/F})$; 
\item[$\bullet$] $\ago_P^{*}= X^{*}(P)_F\otimes\mathbb{R}$; $\ago_P= \Hom_{\mathbb{R}}(\ago_P^{*}, \mathbb{R})$; $\ago_{P, \mathbb{Q}}= \Hom_{\mathbb{R}}(\ago_P^{*}, \mathbb{Q})$; 
\end{enumerate}
Let $A_P$ be the maximal split central torus defined over $F$ of $P$. We should think of $X^*(P)\subseteq X^*(A_P)$ as lattices in $\ago_P^{*}$ and think of $\ago_P^{*}$ as the space of linear functions on $\ago_P$.  

Now suppose that $P, Q$ are standard parabolic subgroups of $G$. We will use the following notations. 
\begin{enumerate}
\item[$\bullet$] $\langle \cdot, \cdot \rangle:$ the canonical pairing between $\ago_B^{*}$ and $\ago_B$;
\item[$\bullet$] $\ago_P=\ago_P^Q\oplus \ago_Q$ ($P\subseteq Q$), decomposition induced by $X^{*}(Q)\rightarrow X^{*}(P)$ and $X^{*}(A_P)\rightarrow X^{*}(A_Q)  $; 
\item[$\bullet$] $\Phi(P,A)$: roots of the torus $A$ acting on the Lie algebra of $\ppp$;
\item[$\bullet$] $\Delta_B$ simple roots in $ \Phi(B, A)$; 
\item[$\bullet$] $\Delta_B^P=\Delta_B\cap \Phi(M_P,A);$ $(\Delta_B^{P})^{\vee}$ the set of corresponding coroots; $\hat{\Delta}_B^P$ the set of corresponding fundamental weights; 
\item[$\bullet$] $\Delta_P^Q=\{\alpha|_{A_P} \mid \alpha \in \Delta_B^Q-\Delta_B^{P}\}$, viewed as a set of linear functions on $\ago_B$ via the projection $\ago_B\rightarrow \ago_P^Q$;
%\item[$\bullet$] $\Delta_P^{\vee}$, $(\Delta_P^{Q})^{ \vee}$ set of coroots corresponding respectively to $\Delta_P$ and $\Delta_P^Q$;
\item[$\bullet$] $\hat{\Delta}_P^Q=\{\varpi \in \hat{\Delta}_B^Q  \mid \varpi|_{\ago_B^P}=0  \}$;
\item[$\bullet$] $\htau_P^{Q}$ the characteristic function of $H\in \ago_B$ such that $\langle\varpi,H\rangle>0$ for all $\varpi\in \hat{\Delta}_P^Q$; 
\item[$\bullet$] $\tau_P^{Q}$ the characteristic function of $H\in \ago_B$ such that $\langle\alpha,H\rangle>0$ for all $\alpha\in \Delta_P^Q$; 
\end{enumerate}

As a general rule of notations, for a given place $\infty$ of $F$,
 we will add an index $\infty$ if a relevant notation is defined using datum over $F_\infty$, for example we have
\begin{enumerate}
\item[$\bullet$] $W_\infty$, $\ago_{P, \infty}$, $\Delta_{B, \infty}$, ... 
\end{enumerate}

\subsection{Harish-Chandra's map}
For each semi-standard parabolic subgroup $P$ of $G$ defined over $F$, we define the Harish-Chandra's map $H_P: P(\AAA)\rightarrow \ago_{P} $ by requiring   
\begin{equation}q^{\langle\alpha,  H_P(p)\rangle} = |\alpha(p)|_\AAA,  \quad \forall \alpha\in X^{*}(P)_F; \end{equation}
Using Iwasawa decomposition we may extend the definition of $H_P$ to $G(\AAA)$ by asking it to be $G(\ooo)$-right invariant. 

\subsection{Haar measures} 
The Haar measures are normalized in the following way.
For any place $v$ of $F$ and 
for a Lie algebra $\ggg$ defined over $\mathbb{F}_q$, the volumes of the sets $\ggg(\ooo_v)$ and $\ggg(\ooo)$ are normalized to be $1$; the volume of $G(\ooo_v)$ and $G(\ooo)$ are normalized to be $1$; for every semi-standard parabolic subgroup $P$ of $G$ defined over $\mathbb{F}_q$, the measure on $N_P(\AAA)$ is normalized so that $\vol(N_P(F)\backslash N_P(\AAA))=1$.

\section{Characteristic of the base field and unipotent/nilpotent elements}\label{alggr}

\subsection{}\label{chara}
In this article, we will treat separately the case of a group and the case of a Lie algebra, according to that we suppose that the hypothesis $(\ast)_G$ or the hypothesis $(\ast)_{\ggg}$ is satisfied. 
\begin{center}
\textit{\((\ast)_G\)
The characteristic does not divide $|\pi_1(G^{der})|$.}
\end{center}
\begin{center}
\textit{\((\ast)_\ggg\)
The characteristic is very good for $G$ or $G=GL_n$. The characteristic $p$ does not divide the degree of the minimal splitting field of $G$.}
\end{center}

We recall that if $G$ is simple, then the characteristic $p$ is said to be very good if $p\neq 2$ when $G$ is one of type $B,C,D$,  $p\neq 2,3$ if $G$ is of one of type $E,F,G$, and $p\neq 2,3,5$ if $G$ is of type $E_8$. For type $A_{n-1}$, we require that $p\nmid n$. In general, a prime is very good for $G$ if it's very good for every simple factor of $G^{sc}_{\overline{F}}$. 
For more discussion of ``very good primes", see \cite[I.4]{SS} or \cite[2.1]{McN}. In particular, $(\ast)_\ggg$ implies $(\ast)_G$. 
% Note the characteristic is very good for $SL_{n,\mathbb{F}_q}$ if and only if $p\nmid n$. For a group without type $A$ factor, all primes strictly larger than $5$ are very good. 
%\begin{remark}
%Pretty good primes is enough.
%\end{remark}

The following Proposition is well known to experts in algebraic groups over positive characteristic.
\begin{prop}\label{ch}
Let $G$ be a reductive group defined over $\mathbb{F}_q$ so that $G$. Then
\begin{enumerate}
\item  \label {1.} If $(\ast)_G$ is satisfied, the central isogeny  $G^{sc}\rightarrow G^{der}$ is smooth (equivalently is separable in older terminology).
\item \label{2.} If $(\ast)_\ggg$ is satisfied, 
the Lie algebra $\ggg$ admits a non-degenerate $G$-invariant bilinear form defined over $\mathbb{F}_q$.
 \end{enumerate}
\end{prop}
\begin{proof}
%The case of $G=GL_n$ is trivial, as the trace form is a non-degenerate $G$-invariant bilinear form defined over $\mathbb{F}_q$ and the morphism defined by $l(X)=X-1$ is a required morphism between unipotent variety and nilpotent variety. 
%Now suppose that the characteristic is very good for $G$. 

%The first assertion follows since $p$ is very good implies that $p\nmid \# Z_{G^{sc}}(\bar{\mathbb{F}}_q)$, thus the scheme theoretic kernel of ${G^{sc}}\rightarrow {G^{der}}$ is a finite group scheme of order prime to $p$, which must be smooth. 

The fundamental group $\pi_1(G^{der})$ is the kernel of $G^{sc}\rightarrow G^{der}$. It is étale since $p$ does not divide the order of $\pi_1(G^{der})$, hence $G^{sc}\rightarrow G^{der}$ is smooth.

If $G$ is semi-simple and simply connected, the assertion \ref{2.} is proved in \cite[Lemma 1.8.12]{KV} (the hypothesis that the characteristic $p$ does not divide the degree of the minimal splitting field of $G$ ensures that the bilinear form in $loc.$ $cit.$ is non-zero).
In general, consider the central isogeny $Z\times G^{sc}\rightarrow G$, where $Z$ is the maximal central torus of $G$, it is separable since $Z\cap G^{sc} $ is a subgroup of $ Z_{G^{sc}}$ which has order prime to $p$. Therefore, we have a $G$-equivariant isomorphism of Lie algebras
\( \ggg\cong\ggg^{sc}\oplus \zzz. \)
We can extend a non-degenerate bilinear form on $\ggg^{sc}$ to $\ggg$ by choosing an arbitrary non-degenerate bilinear form on $ \zzz$.

\end{proof}

We know that every semi-simple element in $G(F)$ is contained in the set of $F$-points of some maximal torus defined over $F$.
We also know that a unipotent (resp. nilpotent) element in $G(F)$ (resp. $\ggg(F)$) is contained in the unipotent radical (resp. the Lie algebra of the unipotent radical) of a parabolic subgroup defined over $\bar{F}$. In positive characteristic, it is not always true that one can choose such a parabolic subgroup defined over $F$. We will need the following result of P. Gille which is applicable for all fields $F$ such that $[F:F^p]\leq p$, in particular for a global function field. 

%Note that the condition $[F:F^p]\leq p$ is satisfied by a global function field and its completions. If one excludes a larger set of primes (i.e.  torsion primes), then the following result is proven by Tits, but we would like that for $GL(n)$, all the primes can work. 

\begin{prop}[Gille, \cite{Gille}]\label{Gille}
Let $G$ be a reductive group so that the hypothesis $(\ast)_G$ in the \ref{chara} is satisfied. Then every unipotent element  $u\in G(F)$ is contained in the unipotent radical of some parabolic subgroup of $G$ defined over $F$. 
\end{prop}
\begin{proof}
The theorem is proved by Gille when $[F:F^p]\leq p$ for semi-simple and simply connected groups (\cite[Theorem 2, p. 313]{Gille}). %Using the terminology of the $loc.$ $cit.$, a unipotent element $u\in G(F)$ is called good if it's contained in the unipotent radical of some parabolic subgroup defined over $F$. Let $H\rightarrow G$ be a central morphism of algebraic groups. Then the morphism $H(F)\rightarrow G(F)$ induces a bijection between good unipotent elements.

In general, let $G^{sc}\rightarrow G^{der}$ be the simply connected covering of $G^{der}$ which is of degree prime to $p$ by our assumption, then any unipotent element in $G^{der}(F)$ is contained in the unipotent radical of some parabolic subgroup defined over $F$ by Proposition (ii) and (iii) of page 267 of \cite{Tits2}.
In fact, we have an exact sequence:
\[ 1\longrightarrow \pi_1(G^{der})(F) \longrightarrow G^{sc}(F)\longrightarrow G^{der}(F)\longrightarrow H^{1}(F, \pi_1(G^{der})). \]
Since $p\nmid | \pi_1(G^{der})|$ and unipotent elements have $p$-power orders, there is a bijection between unipotent elements in $G^{sc}(F)$ and $G^{der}(F)$. The rest is clear.

Now consider the exact sequence: 
 \[1\longrightarrow G^{der}(F)\longrightarrow G(F)\longrightarrow (G/G^{der})(F). \] The image of a unipotent element $u\in G(F)$ in $(G/G^{der})(F)$  is the identity as it's a unipotent element in the torus $G/G^{der}$. Hence $u$ belongs to $G^{der}(F)$. We know that $P\mapsto P\cap G^{der}$ defines a bijection between the set of parabolic subgroups of $G$ and that of $G^{der}$ (Proposition 5.3.4 \cite{Conrad}), such that the unipotent radical $N_{P\cap G^{der}}$ of $P\cap G^{der}$ in $G^{der}$ equals to $N_P\cap G^{der}$ (Proposition. 4.1.10.(2) of \cite{Conrad}). As $u$ is contained in the unipotent radical of an $F$-parabolic subgroup of $G^{der}$,  it is in the unipotent radical of the corresponding parabolic subgroup of $G$.

\end{proof}

\begin{prop}[McNinch, \cite{McN}] \label{McNinch}
Let $G$ be a reductive group so that the hypothesis $(\ast)_\ggg$ in the \ref{chara} is satisfied. 
Let $X\in \ggg(F)$ be a nilpotent element, then it's contained in the Lie algebra of the unipotent radical of a parabolic subgroup defined over $F$.
\end{prop}
\begin{proof}
Under our hypotheses, there is a cocharacter $\phi$ associated to X which is defined over $F$, i.e. $\lambda: \mathbb{G}_{m/F} \rightarrow G$ such that $\Ad(\phi(t))X=t^{2}X$ for all $t\in \mathbb{G}_m$ (\cite[Theorem 26 and Proposition 6]{McN}). This cocharacter defines the $F$-parabolic subgroup $P_G(\lambda):= \{g\in G\mid \lim_{t\rightarrow 0} \lambda(t)g\lambda(t)^{-1} \text{ exists}\}$ of $G$ such that $X$ lies in the Lie algebra of the unipotent radical $U_G(\lambda):= \{g\in G\mid \lim_{t\rightarrow 0} \lambda(t)g\lambda(t)^{-1}=1\}$ of $P(\lambda)$. 
\end{proof}

\section{Combinatorics of roots and coroots under field extension} \label{u}
\subsection{}
When studying reduction theory of adèles, we need some results concerning the combinatorics of roots and coroots under field extensions. 

Let $\infty$ be a place of $F$. Using the inclusion $X^{*}(T)_{F}\hookrightarrow X^{*}(T)_{F_\infty}$, we may view $\ago_B^{*}$ as a subspace of $\ago_{B, \infty}^{*}$. Taking dual of this embedding, we get a surjection
\begin{equation}\label{41}\pr_{\infty}:\ago_{B,\infty}\longrightarrow \ago_B.\end{equation}
Let's denote $\pr_\infty(H)$ by $[H]$. We define the local Harish-Chandra's function $H_{B, \infty}: G(F_\infty)\rightarrow \ago_{B, \infty} $ by requiring that for any $x=p k$ with $p\in B(F_\infty)$ and $k\in G(\ooo_\infty)$, one has 
\begin{equation} \label{HC} |\kappa_\infty|^{\langle\alpha,  H_{B,\infty}(x)\rangle_{\infty}} = |\alpha(p)|_\infty,  \quad \forall \alpha\in X^{*}(B)_{F_\infty}. \end{equation}

\begin{prop}\label{4.1}
Using the embedding $G(F_\infty)\hookrightarrow G(\AAA)$, the relation between global and local Harish-Chandra's map are given by 
\begin{equation}
H_B(g_\infty)=[\kappa_\infty:\mathbb{F}_q]  [H_{B,\infty}(g_\infty)], \quad \forall g_\infty\in G(F_\infty). 
\end{equation}
\end{prop}
\begin{proof}
It follows directly from the definitions.
\end{proof}

Let $A$ be the maximal $F$-split subtorus of $T$ and $A'$ be the maximal $F_\infty$-split subtorus of $T_{F_\infty}$. Clearly $A_{F_\infty}\subseteq A'$ and $A'$ is in fact defined over the residual field $\kappa_\infty$ of $F_\infty$.
As $G$ is quasi-split over $F$, equivalently the centralizer of $A$ in $G$ is the torus $T$, a non-zero root $\alpha$ for the adjoint acton of ${A'}$ on $\ggg_{F_\infty}$ is non-zero after restriction to $A$. 
We thus have a surjective map of the root systems \begin{equation}\Phi(G,A')_{F_\infty} \longrightarrow \Phi(G, A)_{F_\infty}\cong \Phi(G,A)_F.\end{equation}

\begin{theorem}\label{proj}
 Let $\alpha'\in \Phi(G, {A'})$.
The restriction $\alpha'|_{A}$ is an element in $\Phi({G,A})$ and we have
\[ [{\alpha'}^{\vee}]=c_{\alpha'}(\alpha'|_{A})^{\vee}, \]
for some $1\geq c_{\alpha'}>0$. 
% Let $\alpha_0\in \Phi(G,A)$, and $\Delta_{\alpha_0}$ be the fiber of the map $\Phi(G,{A'})\rightarrow \Phi(G,A)$ in $\alpha_0$. Then $$s_{\alpha_0}=\prod_{\alpha\in \Delta_{\alpha_0}} s_{\alpha},$$ where the product is taken in an arbitrary order. 
%\item Let $w\in W\subseteq W_u$, and $H\in \ago_{B_u}$.  We have $$\pr_u(w(H))=w\pr_u(H).$$ 
\end{theorem}
\begin{proof}
There is no harm to choose $\bar{\kappa}_\infty$ to be $\bar{\mathbb{F}}_q$ by fixing an embedding $\kappa_\infty \hookrightarrow \bar{\mathbb{F}}_q$. 
Suppose $\kappa_\infty\cong \mathbb{F}_{q^d}$.
Let $\Gamma_d=\Gal(\mathbb{F}_{q^d}|\mathbb{F}_q) $. It acts on $X^{*}(T)_{F_\infty}\cong X^{*}(T)_{\mathbb{F}_{q^d}}$. 
Note that $\Gamma_d$ preserves $\Phi'=\Phi(G,A')_{F_\infty}$ and is compatible with the natural action of $W(\Phi')$, the Weyl group of $\Phi'$, i.e. 
     \( {}^{\sigma}(w\chi')={{}^{\sigma}w}   \leftidx{_{}^{\sigma}}{\chi'}, \)  
for any  $\sigma\in \Gamma_d$, any $ w\in W(\Phi')$, and any $ \chi'\in X^{*}(A')_{F_\infty} $. Thus $\ago_{B, \infty}^{*}$ admits a positive definite symmetric bilinear $(,)_{\infty}$ which is invariant under the Weyl group $W(\Phi')$ and the Galois group $\Gamma_d$. 

By definition,  for any $\chi\in \ago_B^{*}$, one has \(  \langle \chi, [\alpha'^{\vee}]\rangle=\langle \chi, \alpha'^{\vee}\rangle_{\infty},\) hence \begin{equation}\label{ip}\langle \chi, [\alpha'^{\vee}]\rangle=\frac{2(\alpha', \chi)_\infty}{(\alpha',\alpha')_\infty}. \end{equation}

Since $A_{F_\infty}$ is a sub-torus of $A'$, we have a quotient map \[ X^{*}({A'})_{F_\infty}\longrightarrow X^{*}(A)_{F_\infty}\cong X^{*}(A)_{F}. \] which defines a projection via the inclusions  $X^{*}(T)_F \subseteq X^{*}(A)_F$ and $X^{*}(T)_{F_\infty}\subseteq X^{*}(A')_{F_\infty}$:
\[ \pr_\infty^{*}: \ago_{B,\infty}^{*}\longrightarrow  \ago_B^{*}. \]
%Let $d=[\kappa_\infty: \mathbb{F}_q]$ be the degree of $\infty$. We use $\tau$ to denote the Frobenius automorphism $\bar{\mathbb{F}}_{q}|\mathbb{F}_q$.  Suppose that the fixed maximal torus $T$ splits over a degree $r$ extension of $\mathbb{F}_q$.  Then the torus $T_{\kappa_\infty}$ splits over a degree $r/(r,d)$ extension of $\kappa_\infty$.  Let $A$ be the maximal $\mathbb{F}_q$-split subtorus of $T$. The torus $A$ is not necessarily maximal $\kappa_\infty$-split subtorus of $T$.  Let $A'$ be the unique subtorus of $T_{\bar{\mathbb{F}}_q}$ defined over $\bar{\mathbb{F}}_q$ whose absolute character group is given by  $$X^{*}(T)_{\bar{\kappa}_\infty}/\{\chi\in X^{*}(T)_{\bar{\kappa}_\infty} |\sum_{i=1}^{r/(r,d)} {}^{\tau^{di}}\chi=0 \}.$$ Then $A'$ is in fact defined over $\mathbb{F}_q$ and it splits over $\kappa_\infty$. Thus $A\subseteq A'$. 
We leave it to the reader to verify that for any $\chi'\in X^{*}(A')_{F_\infty}$, we have  \[ \pr_\infty^{*}(\chi')= \chi'|_{A}= \frac{1}{d}\sum_{\sigma\in \Gamma_d} {}^{\sigma}\chi'. \]
Hence $\pr_\infty^{*}$ is an orthogonal projection and for any $\chi\in \ago_B^{*}$, we have \[ (\alpha', \chi)_\infty=(\pr_\infty^{*}(\alpha'),\chi)_\infty=(\alpha'|_A, \chi)_\infty.  \]
We may rewrite (\ref{ip}) as \[ \langle \chi, \frac{1}{c_{\alpha'}}[\alpha'^{\vee}]\rangle = \frac{2(\alpha'|_{A}, \chi)_\infty}{(\alpha'|_A,\alpha'|_A)_\infty}, \]
with $c_{\alpha'}=\frac{(\alpha'|_A,\alpha'|_A)_\infty}{(\alpha', \alpha')_\infty}$, which implies the assertion. 
\end{proof}

\subsection{Complementary polyhedra}
The definition of complementary polyhedra is introduced by Behrend in \cite{Behrend} and by Arthur in different languages (see the remark below). 
\begin{definition}
Let $(X_s)_{s\in W}$ be a family of points in  $\ago_B$ parametrized by elements of the Weyl group of $G$.  
It's called a complementary polyhedron (for the root system $\Phi(G,A)$) if for any $s=s_\alpha t$, where $s_\alpha$ is the symmetry with respect to the simple root $\alpha\in \Delta_B$, one has 
\[ X_{t}-X_{s}= b_{\gamma}(s,t)\gamma^{\vee} , \]
with $b_{\gamma}(s,t)\geq 0$ and $\gamma= s^{-1}\alpha=-t^{-1}\alpha$. 
\end{definition}

\begin{remark}\label{Q}
1. If $(X_s)_{s\in W}$ is a complementary polyhedron in $\ago_B$ for the root system $\Phi(G,A)$, then for any parabolic subgroup $Q\in \mathcal{P}(T)$, $(X_s)_{s\in W^Q}$ is a complementary polyhedron in $\ago_B\cong \ago_{B\cap M_Q}$ for the root system $\Phi(M_Q, A)$. \\
2.  Arthur calls $(X_s)_{s\in W}$ a regular orthogonal family if the $b_\gamma(s,t)$ above are strictly negative (compare \cite[Section I.6]{LabWal}).  
\end{remark}

Recall that $W=W^{(G,T)}(\mathbb{F}_q)$, and $W_{\infty}\cong W^{(G,T)}(\kappa_\infty)$. Thus $W$ is identified with a subgroup of $W_\infty$.

\begin{prop}\label{projcp}
Let $(X_s)_{s\in W_{\infty}}$ be a complementary polyhedron in $\ago_{B,\infty}$ for the root system $\Phi(G, A')$,  then $([X_s])_{s\in W}$ is a complementary polyhedron in $\ago_B$  for $\Phi(G, A)$. 
\end{prop}
\begin{proof}
Let $\alpha\in \Delta_B$ be a simple root, and $s, t\in W$ such that $s=s_\alpha t$. 
Then (by \cite[1.5.1]{LabWal}) there are non-negative numbers $b_{\beta}'(s,t)$ depending on the choice of a reduced decomposition of $s_{\alpha}$ into simple symetries in $W_\infty$ such that
\[ X_{s}-X_{t}=\sum_{\{\beta'\in \Phi(G,{A'}) |  t\beta'>0, s \beta'<0 \}} b_{\beta'}(s,t)    \beta'^{\vee}.  \]

Note that $\beta'>0$ if and only if $\beta'|_{A}>0$, and the only root $\beta\in \Phi(G,A)$ such that $t\beta>0$  but $s\beta<0$ is $t^{-1}\alpha$. Hence those $\beta'$ which show up in the sum above are those which satisfy 
$\beta'|_A=t^{-1}\alpha$. By Theorem \ref{proj}, 
\[ [X_{s}]-[X_{t}]= b_{t^{-1}\alpha}(s,t)   (t^{-1}\alpha)^{\vee},  \]
where \begin{equation}\label{rsta} b_{t^{-1}\alpha}(s,t)=\sum_{ \beta'|_A=t^{-1}\alpha}  c_{\beta'}b_{\beta'}(s,t). \end{equation}  %where $ c_{t^{-1}\beta'}$ are numbers in Theorem \ref{proj}.
\end{proof}
%Now we fix a place $\infty$ of $F$. For a vector $X\in \ago_B$, we define its height $d(X)$ by \[d(X)=\min_{\alpha\in \Delta_B}\langle \alpha, X\rangle. \]
\begin{definition}\label{admissible} 
We say that a pair of vectors $(\xi, X)\in \ago_{B, \infty}\times \ago_B$ is admissible if 
\( d(X)\geq 0, \) and 
  \[ - \frac{1}{[\kappa_\infty:\mathbb{F}_q]} d(X)  \leq  \langle\alpha,  \xi  \rangle_{\infty}  \leq \frac{1}{[\kappa_\infty:\mathbb{F}_q]} d(X) + {[\kappa_\infty: \mathbb{F}_q]}, \quad \forall \alpha\in \Phi(G_{F_\infty}, A')^{r}_{+}, \]
where $d(X):=\min_{\alpha\in \Delta_B}\alpha(X)$, and $\Phi(G_{F_\infty}, A')^{r}_{+}$ is the set of positive roots in the reduced root system. 
  \end{definition}

For any $x\in  G(\mathbb{A})$. We have an Iwasawa decomposition: $x = pk$, with $p\in B(\AAA)$  and $k\in G(\mathcal{O})$. The element  $k$ is unique up to left translation by $B(\AAA)\cap G(\mathcal{O}) = B(\mathcal{O})$, hence determines an element in \[ \mathcal{I} \backslash G(\mathcal{O}_{\infty})/ \mathcal{I}  \cong B({\kappa}_\infty)\backslash G({\kappa}_\infty)/B({\kappa}_\infty)\cong W_{\infty}, \] where $\mathcal{I}$ is the Iwahori subgroup of $G(F_\infty)$ with respect to $B$, i.e. group of elements in $G(\ooo_\infty)$ whose images in $G(\kappa_\infty)$ are contained in $B(\kappa_\infty)$. 
We denote this Weyl element by $s_{x}$.

We have the following proposition. In fact, for the proofs of our main theorems we only need it in the case when $d(X)\gg 0$ depending on $\xi$, which is a direct consequence of \cite[Lemme 3.3.2]{LabWal}. However, it's interesting in itself to have an explicit bound since there might be some further applications that require $X=0$.   
\begin{theorem}\label{complementary} 
Let $(\xi, X)\in \ago_{B, \infty}\times \ago_B$ be admissible. 
For any $h\in G(\AAA)$, we define for each $s\in W$ a vector
\[ X_{s}=s^{-1}(H_{B}(w_s h) + [s_{w_s h}\xi]  -X)\in \ago_{B}. \]
Recall that $w_s$ is a fixed representative  in $N_G(T)(\mathbb{F}_q)$ of $s\in W$. 
 Then the family $(X_{s})_{s\in W}$ is a complementary polyhedron  in $\ago_{B}$.
\end{theorem}
\begin{proof}
%Let $h=(h_v)\in G(\AAA)$. 
%As $H_B(h_v)=[\deg v H_{B,v}(h_v)]$ and the projection of a complementary polyhedron form $\ago_{B, v}$ to $\ago_B$  is a complementary polyhedron (Proposition \ref{projcp}), 
The problem is local in nature. 
It suffices to prove that, for any place $v$ of $F$, any $h\in G(F_v)$, the family 
\( (s^{-1}H_{B}(w_s h))_{s\in W}  \)
is a complementary polyhedron in $\ago_{B}$ and 
for any $h \in G(F_\infty)$,  \[ \left(s^{-1}(H_{B}(w_sh)  + [s_{w_sh}\xi]  -  X ) \right)_{s\in W} \]
is a complementary polyhedron in $\ago_{B}$. Without loss of generality, we prove the second statement only. 

Recall that $A'$ is the maximal $F_\infty$-split subtorus of $T$. 
For any $\gamma\in \Phi(G, A')$, let $N_\gamma$ be the root subgroup of $G_{F_\infty}$ associated to $\gamma$. 
Fix a valuation $(\varphi_\gamma)_{\gamma\in \Phi(G,A')}$ of the root datum (cf. \cite[6.2.1]{BT1} and \cite[4.2]{BT2}) which is centered at the hyperspecial point corresponding to $G(\ooo_\infty)$,  and $\varphi_{\gamma}: N_{\gamma}(F_\infty)\rightarrow \mathbb{Z}\cup \{\infty\}$ is normalized to be surjective. %Since $G$ is unramified over $F_\infty$, we can normalize $\varphi_\gamma$ so that 
Let 
\[N_{\gamma, n}=\varphi_\gamma^{-1}([n,\infty])\subseteq N_{\gamma}(F_\infty).\]

For any $s\in W_\infty$, we can choose a representative in $N_G(T)(\kappa_\infty)$.  
Let $B_1=w_s^{-1}Bw_s$ and $B_2=w_{t}^{-1}Bw_t$ for $s, t\in W_{\infty}$, be a pair of adjacent Borel subgroups defined over $F_\infty$  both containing $T$. We may assume that $s=s_\alpha t$ with $\alpha\in \Delta_{B,\infty}$, then $\Delta_{B_1,\infty}\cap (-\Delta_{B_2, \infty})=\{\gamma\}$, where $\gamma=s^{-1}(\alpha)=-t^{-1}\alpha$.  
We have 
\[ N_{B_1}=(N_{B_1}\cap N_{B_2}) N_{\gamma}. \]
By Iwasawa decomposition, we can write $h=auxk$ with $a\in T(F_\infty)$, $u\in (N_{B_1}\cap N_{B_2})(F_\infty)$, $x\in N_{\gamma}(F_\infty)$ and  $k\in G(\ooo_\infty)$. 
We claim that 
\begin{equation}\label{1i}t^{-1}H_{B}(w_th)-s^{-1}H_{B}(w_sh)=\begin{cases} 0, \quad \text{if } x\in N_{\gamma,0}; \\ -[\kappa_\infty:\mathbb{F}_q]\varphi_{\gamma}(x) \gamma^{\vee} ,\quad x\notin N_{\gamma,0};
 \end{cases}\end{equation} 
and that
 \begin{equation}
 \label{2i}t^{-1} s_{w_t h}\xi - s^{-1} s_{w_s h} \xi =\begin{cases}0, \quad \text{if  $x\in N_{\gamma,0}$ and $s_{w_tk}^{-1}\alpha>0$}; \\ 0 \text{ or }  - \langle s_{w_t k}^{-1} \alpha,  \xi\rangle_{\infty} \gamma^{\vee}, \quad \text{if  $x\in N_{\gamma, 0}$ and $ s_{w_tk}^{-1}\alpha<0 $}  ;\\
- \langle s_{w_t k}^{-1} \alpha, \xi\rangle_{\infty} \gamma^{\vee} , \quad \text{if  $x\notin N_{\gamma,0}$}.  \end{cases} \end{equation} 
%The family $(-s^{-1}X)_{s\in W}$ is a complementary polyhedron, as for any $s\in \Delta_{B}$, \begin{equation}\label{3i} s^{-1}X-t^{-1} X=  \langle \alpha, X\rangle \gamma^{\vee}. \end{equation} Combining (\ref{1i}) (\ref{2i}) (\ref{3i}), and the fact that $(\xi, X)$ is admissible, we can conclude now $$(t^{-1}H_B(w_t h)- s^{-1}H_B(w_s h)) +( s^{-1}X- t^{-1} X ) +  (s_{h}\xi -s_\alpha s_{w_\alpha h}\xi) = c\gamma^{\vee} $$ for some   $c\geq 0$.
 
These assertions suffice to prove the theorem. In fact, if $x\in N_{\gamma, 0}$, then the sum of (\ref{1i}) and (\ref{2i}) is either $0$ or of the form $c\gamma^{\vee}$ with $c\geq -\frac{1}{[\kappa_\infty: \mathbb{F}_q]} d(X)$. If $x\notin N_{\gamma, 0}$, then since \[ -[\kappa_\infty:\mathbb{F}_q]\varphi_{\gamma}(x)\geq  [\kappa_\infty:\mathbb{F}_q], \] and  by the definition of admissible pairs (Definition \ref{admissible}), if $-s_{w_t k}^{-1} \alpha$ is positive, we have 
\[ \langle -s_{w_t k}^{-1} \alpha, \xi\rangle_{\infty} \geq -\frac{1}{[\kappa_\infty: \mathbb{F}_q]} d(X).\] 
Otherwise $s_{w_t k}^{-1} \alpha$ is positive, we have \[-\langle s_{w_t k}^{-1} \alpha, \xi\rangle_{\infty} \geq -\frac{1}{[\kappa_\infty: \mathbb{F}_q]} d(X)- {[\kappa_\infty: \mathbb{F}_q]}. \]
The sum of (\ref{1i}) and (\ref{2i}) is always of the form $c\gamma^\vee$ with $c\geq -\frac{1}{[\kappa_\infty: \mathbb{F}_q]} d(X)$. Note that for each $\alpha_0\in \Delta_B$, there are at most $[\kappa_\infty, \mathbb{F}_q]$ roots $\alpha\in \Delta_{B,\infty}$ such that $\alpha|_A=\alpha_0$ (they form one orbit for a Galois action).  
The statement of the theorem is then a corollary of \eqref{rsta}, (\ref{1i}) and (\ref{2i}). 
  We prove (\ref{1i}) and (\ref{2i}) in the following.

For \eqref{1i}, a proof for a split reductive group is written down in \cite[8.16.8]{C-L1}. The proof is similar in our case.
For any $h\in G(F_\infty)$ we have \[ s^{-1}H_{B, \infty}(w_s h)=H_{w_s^{-1}Bw_s, \infty}(h). \]
We always have $H_{B_2,\infty}(h)=H_{B_2,\infty}(a)$ and \[t^{-1}   H_{B,\infty}(w_th)- s^{-1}    H_{B,\infty}(w_sh) =H_{B_2,\infty}(x)=s_{\gamma}^{-1}H_{B_1,\infty}(w_\gamma x).  \]

If $x\in N_{\gamma, 0}$, then clearly $H_{B_2,\infty}(x)=0$. Otherwise let $m=\varphi_{\gamma}(x)$ then
$m<0$.
Let \[M_{\gamma, m}=  M_{\gamma}\cap N_{-\gamma}(F_\infty)\varphi_\gamma^{-1}(m)N_{-\gamma}(F_\infty), \] 
where in our case $M_{\gamma}=w_\gamma T(F_\infty)$. 
We have (\cite[6.2.2. (3)]{BT1})
    \begin{equation}\label{cool}
    w_\gamma \varphi_{\gamma}^{-1}(m)  \subseteq w_{\gamma}N_{-\gamma, -m}  M_{\gamma, m} N_{-\gamma, -m}= N_{\gamma, -m} (w_{\gamma} M_{\gamma,m}) N_{-\gamma, -m}. \end{equation}
It follows that $H_{B_1, \infty}(w_\gamma x)=H_{B_1, \infty}(w_{\gamma}n)$ for some $n\in M_{\gamma,m}$.  
   As $M_{\gamma,m}\subseteq G^{{der}}(F_\infty)$, we deduce from \cite[6.2.10(ii)]{BT1} that $H_{B_1, \infty}(w_\gamma M_{\gamma, m})= m  \gamma^{\vee}$. The equality \eqref{1i} then follows from the Proposition \ref{4.1}.
   
   % the image of any element of $w_\gamma M_{\gamma,m}$ under $H_{B,\infty}$ is  the translation by $m  \gamma^{\vee}$ (cf. \cite[6.2.10(ii)]{BT1}). 

   Now we prove \eqref{2i}. Note that we have \[ w_t (N_{B_2}\cap N_{B_1})w^{-1}_t\subset B, \text{ and }  w_s (N_{B_2}\cap N_{B_1})w^{-1}_s\subset B , \] hence
 \[ s_{w_t h}=s_{w_t xk}, \text{ and }  s_{w_s h}=s_{w_s xk}  . \] 
   Note that we always have 
   \[ \mathcal{I}w_\alpha \mathcal{I} w_t {k}\mathcal{I}\subseteq \mathcal{I}w_\alpha s_{w_t k} \mathcal{I}\cup \mathcal{I}s_{w_t k}\mathcal{I}.   \]
%If $x\in N_{\gamma, 0}$, then \[w_txw_t^{-1}\in \mathcal{I}.\]  We have \[s_{w_t xk}=s_{w_tk}, \] and \[s_{w_s xk}=s_{w_t k} \text{ or }  s_\alpha s_{w_tk}, \]

If $x\in N_{\gamma, 0}$, then \[w_txw_t^{-1}\in \mathcal{I}.\] 
We have \(s_{w_t xk}=s_{w_tk}. \)
If moreover, $s_{w_tk}^{-1}\alpha>0$, then $l(s_\alpha s_{w_tk})=l(s_{w_tk})+1$ (cf. \cite[1.3.1]{LabWal}) hence we have (cf. Théorème 2, $§2$, IV  of \cite{Bourbaki}) \[\mathcal{I}w_\alpha\mathcal{I}w_t k\mathcal{I}=\mathcal{I}w_{\alpha}s_{w_t k}\mathcal{I}.\]
Otherwise if  $s_{w_tk}^{-1}\alpha<0$, as $\mathcal{I}w_\alpha \mathcal{I} w_t {k}\mathcal{I}\subseteq \mathcal{I}w_\alpha s_{w_t k} \mathcal{I}\cup \mathcal{I}s_{w_t k}\mathcal{I}$, we deduce by definition that:
\[ s_{w_s xk}=\begin{cases} s_\alpha s_{w_tk}, \quad s_{w_tk}^{-1}\alpha>0; 
  \\  s_\alpha s_{w_tk} \text{ or } s_{w_tk},   \quad s_{w_tk}^{-1}\alpha<0. 
\end{cases} \]

If $x\notin N_{\gamma, 0}$, then there is an $m<0$ such that $x\in N_{\gamma, m}$. 
As $N_{-\alpha, -m}\subseteq \mathcal{I}$ whenever $m<0$, by \eqref{cool}, we have $ s_{w_\alpha w_txh}=s_{w_tk}$ and $s_{w_txk}=s_{w_tk}$. 

The equality  (\ref{2i}) follows from these calculations. 
\end{proof}

\section{Statements of the main theorems}\label{maintrace}

\subsection{}
Let $f\in \mathcal{C}_c^{\infty}(G(\mathbb{A}))$ be a complex smooth function with compact support over $G(\AAA)$. Let $Q$ be a standard parabolic subgroup  of $G$ defined over $F$, the kernel function of $f$ acting on  $L^{2}(N_Q(\AAA) M_Q(F)\backslash G(\mathbb{A})/\Xi_G)$ is given by: 
\begin{equation}k_Q(x,y)=\sum_{a\in \Xi_G} \sum_{\gamma\in M_Q(F)} \int_{N_Q(\AAA)} f(a  y^{-1} \gamma  u x)  \d u. \end{equation}
For any $x\in N_Q(\AAA) M_Q(F)\backslash G(\mathbb{A})/\Xi_G$, we will note  \begin{equation}k_Q(x)=k_Q(x, x) . \end{equation}

In the case of Lie algebras, a similar definition is given by analogy. Let $f\in \mathcal{C}_c^{\infty}(\ggg(\mathbb{A}))$ and  $Q$ be a standard parabolic subgroup, we define $\kkk_Q(x)$ for any $x\in G(\AAA)$ by:
 \begin{equation}\kkk_Q(x)= \sum_{\Gamma\in {\mmm_Q}(F)} \int_{\nnn_Q(\AAA)} f( \mathrm{Ad}(x^{-1})(\Gamma+ U)) \d U. \end{equation}

%%%%%%%%%%%%%%%%%%%%%%%%%%%%%%%%%%%%%%%%%

We still fix a place $\infty$ of $F$. 
For any vector $\xi \in \ago_{B, \infty}$ and $X\in \ago_B$, 
we define a variant of Arthur's truncated kernel  by:
\begin{equation}  {k}^{\xi, X}(x) =\sum_{Q\in \mathcal{P}(B)}(-1)^{\dim\ago_Q^G}\sum_{\delta\in Q(F)\backslash G(F)}  \htau_{Q}\left(H_B(\delta x) +  [s_{\delta x}\xi] -X  \right) {k}_{Q}(\delta x), \end{equation}
for any $x\in G(\AAA)$. Recall that $\mathcal{P}(B)$ is the set of standard parabolic subgroups and $[s_{\delta x}\xi]$ is defined in the previous section ($[\cdot]$ is defined by the map \eqref{41} and $s_{x}$ is defined before the Theorem \ref{complementary}).  
In the Lie algebra case, we define similarly:
\begin{equation}  \kkk^{\xi, X}(x)=\sum_{Q\in \mathcal{P}(B)}(-1)^{\dim\ago_Q^G}\sum_{\delta\in Q(F)\backslash G(F)}  \htau_{Q}\left(H_B(\delta x) +  [s_{\delta x}\xi]  -X \right) \kkk_{Q}(\delta x), \end{equation}
for any $x\in G(\AAA)$ and $(\xi, X)\in\ago_{B,\infty}\times \ago_B$.  
One can see that, for every $\delta'\in Q(F)$, the elements $s_{\delta x}$ and $s_{\delta'\delta x}$ represent the same coset in $W^Q_\infty\backslash W_\infty$, hence $s_{\delta x}\xi $ and $s_{\delta'\delta x}\xi $ have the same projection in $\ago_Q$. Besides, for each $x\in G(\AAA)$ and each function $f$ with compact support, the definition only involves a finite sum (\cite[3.7.1]{LabWal}).
This shows that the definition makes sense. 
 
 As there is a strong analogy between two cases, we will only present one case if the other case follows by the same arguments. 

\subsection{}
We define two elements $\gamma$ and $\gamma'$ in $G(F)$ (resp. in $\ggg(F)$) to be equivalent if either both of them admit Jordan-Chevalley decomposition and their semi-simple parts $\gamma_s$ and $\gamma'_{s}$ are $G(F)$-conjugate, or neither of them admit Jordan-Chevalley decomposition. So an equivalent class is a union of $G(F)$-conjugacy classes and there is one single equivalent class consisting of elements which do not admit Jordan-Chevalley decomposition. 
Let \[\mathcal{E} = \mathcal{E}(G(F)) \text{ or  $\mathcal{E}(\ggg(F))$}\] be the set of equivalent classes for either the group case or the Lie algebra case if it doesn't cause a confusion.

Let $o\in \mathcal{E}$ be an equivalence class. For any standard parabolic subgroup $Q$ of $G$ and $x\in G(\AAA)$, let's define 
\begin{equation}\label{kQo}   k_{Q, o}(x)= \sum_{a\in \Xi_G} \sum_{\gamma\in M_Q(F)\cap {o} } \int_{N_Q(\AAA)} f(ax^{-1} \gamma n x)  \d u, \end{equation}
and 
\begin{equation}  k^{\xi, X}_{o}(x)=\sum_{Q\in \mathcal{P}(B)}(-1)^{\dim\ago_Q^G}\sum_{\delta\in Q(F)\backslash G(F)}  \htau_{Q}\left(H_B(\delta x)+  [s_{\delta x}\xi]  -X \right)k_{Q, o}(\delta x). \end{equation}
It follows from definitions that \begin{equation}\label{geomdec} k^{\xi, X}(x)=\sum_{o\in \mathcal{E} } k^{\xi, X}_{o}(x).\end{equation}
Again, for each $x\in G(\AAA)$ and each compactly supported function $f\in C_c^{\infty}(G(\AAA))$, the sum in (\ref{geomdec}) is in fact a finite sum. 
The same construction applies for the Lie algebra $\ggg$, so we can define $\kkk^{\xi, X}_{o}$.

Our definitions differ from that of Arthur % and L. Lafforgue for a function field for inner forms of $GL_n$ \cite[V.1]{Laf}) 
only by the presence of $\xi$. The following theorem generalizes \cite[Proposition 11, p.227]{Laf} and \cite[Théorème 6.1.1. (2),(3)]{Chau} whose proof will be given later in Section \ref{proofm}.
\begin{theorem}\label{Main}
Assume the hypothesis $(\ast)$ in the Section \ref{chara} on the characteristic is satisfied. 
For any equivalence class $o\in \mathcal{E}$ and any pair of vectors $(\xi, X)\in \ago_{B,\infty}\times \ago_{B}$, 
the function $x\mapsto k^{\xi, X}_{o}(x)$  (resp. $x\mapsto \kkk^{X}_{{\xi}, o}(x)$) has compact support over $G(F)\backslash G(\mathbb{A})/Z_{G}(\AAA)$.

\sloppy
Moreover, when $(\xi,X)\in \ago_{B,\infty}\times\ago_B$ is admissible,
there is a function $F^{\xi, X}(\cdot)$ over $G(F)\backslash G(\AAA)/Z_G(\AAA)$ which is the characteristic function of a set which is compact in $G(F)\backslash G(\AAA)/Z_G(\AAA)$. For every function $f\in \mathcal{C}_c^\infty(G(\AAA))$ (resp. $f\in \mathcal{C}_c^\infty(\ggg(\AAA))$), if 
 $d(X):=\min_{\alpha\in \Delta_B}\alpha(X)$ is large enough depending on $\xi$ and $f$, we have for any $x\in G(\AAA)$
  \begin{align*}k^{\xi, X}_{o}(x) &=F^{\xi, X}(x) k_{G,o}(x) \\
    & =F^{\xi, X}(x)   \sum_{a\in \Xi_G} \sum_{\gamma \in o}f(a x^{-1}\gamma x). \end{align*}
  \[  \text{ (resp.      $\kkk^{\xi, X}_{o}(x) =F^{\xi, X}(x)\sum_{\Gamma \in o}f(\ad(x^{-1})(\Gamma)).$        )}    \]
\end{theorem}	
\begin{remark}
Unlike the above theorem which applies for each $k^{\xi, X}_o(x)$ that requires an additional hypothesis on the characteristic, 
we can prove by the same proof that $k^{\xi, X}(x)=\sum_{o\in \mathcal{E}}k^{\xi, X}_o(x)$ has compact support in $G(F)\backslash G(\AAA)/Z_G(\AAA)$ in any characteristic $p$.  \end{remark}

After this theorem,  the following definitions make sense: for $f\in \mathcal{C}_c^{\infty}(G(\AAA))$ and any $(\xi, X)\in \ago_{B,\infty}\times \ago_{B}$, we define the truncated trace by  
\begin{equation}J^{G, \xi, X}(f):= \int_{G(F)\backslash G(\AAA)/\Xi_G}  k^{\xi, X}(x)\d x;     \end{equation}
and for any $o\in \mathcal{E}$
\begin{equation}J^{G, \xi, X}_{o}(f):= \int_{G(F)\backslash G(\AAA)/\Xi_G}  k^{\xi, X}_{o}(x)\d x;     \end{equation}
Similarly we define $J^{\ggg, \xi, X}(f)$ and $J^{\ggg, \xi, X}_o(f)$ in the Lie algebra case. 
We will omit the superscript $G$ or $\ggg$ if it is clear from the context which case we're dealing with. We will also omit $X$ in the notation if it is set to be zero, which is also our main interests. We have the following corollary of the Theorem \ref{Main}. 
\begin{coro}\label{5.3}
In either the group case, or the Lie algebra case, we have the coarse geometric expansion of the truncated trace:  
\begin{equation}
J^{\xi}(f)=\sum_{o\in \mathcal{E}} J^{\xi}_o(f),
\end{equation}
where there are only finitely many non-zero terms in the sum. 
\end{coro}
\begin{remark}
If the characteristic is not too large, but still satisfies $(\ast)_G$ of the Section \ref{chara}, then the class $o\in \mathcal{E}$ consisting of elements which do not admit Jordan-Chevalley decomposition can be non-empty, hence the corresponding contribution $J^{\xi}_o(f)$ can be non-zero.
\end{remark}

\subsection{Quasi-polynomial behaviour} 
The definition below is taken from \cite[4.5.3]{Chau} (although slightly different). 
\begin{definition}
Let $\Phi: \ago_{B,\mathbb{Q}}\rightarrow \mathbb{C}$ be a function. It's called a quasi-polynomial, if for any lattice $L_0\subseteq \ago_{B,\mathbb{Q}}$, there exist a finite set $\mathfrak{f}\subseteq \frac{2\pi i}{\log q} \ago_{B,\mathbb{Q}}^{*}$,  and a family of polynomials $(p_\nu)_{\nu \in \mathfrak{f}}$ such that for any $X\in L_0$ we have
\[ \Phi(X)=\sum_{\nu \in\mathfrak{f}}p_\nu(X)q^{  \langle{\nu, X}\rangle   } . \]
\end{definition}

The following result is then an analogue of a theorem of Arthur over number fields and a generalization of the case $G=GL_n$ and $\xi=0$ of a result of Chaudouard \cite[Théorème 6.1.1.(4)]{Chau}. 
\begin{theorem}\label{polynomial}
For each $o\in \mathcal{E}$, the maps $X\mapsto J^{G,\xi, X}_{o}(f)$ and $X\mapsto J^{\ggg,\xi, X}_{o}(f)$ are quasi-polynomials. 
\end{theorem}
The proof will be given in the Section \ref{proofm}.

\subsection{A trace formula for Lie algebras}\label{FTLie}
Suppose that our assumption $(*)_{\ggg}$  on characteristic in the Section \ref{chara} is satisfied.

Let $\langle \cdot, \cdot \rangle$ be a $G$-invariant bilinear  form on $\ggg$ defined over $\mathbb{F}_q$, which exists thanks to the assumption on characteristic (Proposition \ref{ch}). Thus, it defines by taking respectively $\AAA$-points,  a $G(\AAA)$-invariant non-degenerate bilinear form on $\ggg(\AAA)$. 

We fix a non-trivial additive character: 
\( \psi: F \backslash \AAA\rightarrow \mathbb{C}^{\times}. \)
For any $f\in \mathcal{C}_{c}^{\infty}(\ggg(\AAA))$, the global Fourier transformation is defined by \begin{equation}\hat{f}(X):=\int_{\ggg(\AAA)}  f(X)\psi(\langle X,Y\rangle)\d Y. \end{equation}

By Poisson summation formula, for any $x\in G(\AAA)$ we have 
\[ \sum_{X\in \ggg(F)} f(\ad(x^{-1})X) = q^{(1-g)\dim \ggg}    \sum_{X\in \ggg(F)} \hat{f}(\ad(x^{-1})X). \]
where the constant $q^{(1-g)\dim \ggg} $ is the inverse of volume of $\ggg(\AAA)/\ggg(F)$ since the measure is normalized so that $\vol(\ggg(\ooo))=1$.

Over a number field, when $\xi$ is trivial, the following formula is obtained already in \cite{ChauLie}.  
\begin{theorem}\label{TFL} Let $J^{\xi}(f)$ be the truncated trace by taking $X=0$.
For any $f\in \mathcal{C}_c^{\infty}(\ggg(\AAA))$, we have
 \[J^{\xi}(f)=q^{(1-g)\dim \ggg} J^{\xi}(\hat{f}). \]
\end{theorem}
\begin{proof}By Theorem \ref{Main}, it follows that \[ \kkk^{\xi, X}(f)(x)=q^{(1-g)\dim \ggg} \kkk^{\xi, X}({\hat{f}})(x), \] when $X$ is deep enough in the positive chamber.
Hence for $X$ deep enough, we have \[ J^{\xi, X}(f)=q^{(1-g)\dim \ggg}  J^{\xi, X}(\hat{f}).\]
As $J^{\xi, X}(f)$ and $J^{\xi, X}(\hat{f})$ are quasi-polynomials (the Theorem \ref{polynomial}), the equality then extends to all $X\in \ago_{B,\mathbb{Q}}$, in particular to $X=0$. 
\end{proof}

\section{Reduction theory and combinatoric lemmas}
In this subsection, we study reduction theory using the notion of complementary polyhedron.

\subsection{$(\xi, X)$-canonical parabolic subgroup}
For any parabolic subgroup $R$ of $G$ defined over $F$, there is a unique  standard parabolic subgroup $Q$, containing the fixed Borel subgroup $B$,  and an element $\eta\in Q(F)\backslash G(F)$ such that $R=\eta^{-1} Q \eta$. We will take advantage of this to denote a parabolic subgroup defined over $F$ by a pair $(Q, \eta)$ if there is no confusion.

\begin{definition}[semi-stability and canonical refinement]\label{sscr}
Let $x \in G(\AAA)$.
A parabolic subgroup $(Q, \eta)$ of $G$ defined over $F$ is called $(\xi, X)$-semi-stable for $x$ if for any $(P, \delta)$ properly contained in $(Q, \eta)$ (i.e. $\delta^{-1} P \delta \subsetneq \eta^{-1}Q\eta$), we have
\[ \hat{\tau}_{P}^{Q}(H_B(\delta x)+ [s_{\delta x}\xi ]-X )=0. \]
The element $x\in G(\AAA)$ itself is called $(\xi, X)$-semi-stable if $G$ is $(\xi, X)$-semi-stable for $x$.

A parabolic subgroup $(P, \delta)$ contained in $(Q,\eta)$ is called a $(\xi, X)$-canonical refinement of $(Q,\eta)$ for $x\in G(\AAA)$ if  \\
 (1.) $(P, \delta)$ is $(\xi, X)$-semi-stable for $x$; \\
 (2.) For any $\alpha\in \Delta_{P}^{Q}$,  \[ \langle\alpha, H_B(\delta x)+ [s_{\delta x}\xi]-X \rangle>0. \]
\end{definition}

The following result is an adelic version of the existence and uniqueness of semi-stable reduction for $G$-bundles with parabolic structures over a curve \cite[Theorem 4.3.2]{HS} based on the work of \cite{Behrend}. 
\begin{theorem}\label{Behrend}
Let $(\xi, X)\in\ago_{B, \infty}\times \ago_B$ be an admissible pair of vectors. 
For any $x\in G(\AAA)$ and any parabolic subgroup $(Q, \eta)$, there exists a unique $(\xi, X)$-canonical refinement of $(Q, \eta)$ for $x$.  
\end{theorem}
\begin{proof}
%Instead of showing how our definitions are connected to those of \cite{HS} and modifying his proof, we give directly a proof under our terminologies (although it's essentially the same proof except that we don't pass to a field that splits $G$).   
Let $P$ be a standard parabolic subgroup of $Q$ defined over $F$. For any $x\in G(\AAA)$, we define 
\begin{equation} \label{degree} \deg_{x}^Q(P) =  \sum_{\alpha\in \Phi(N_{P}\cap M_Q,A)^{r} } \langle \alpha, H_B(x)+ [s_{ x}\xi] -X \rangle, \end{equation}
where $ \Phi(N_{P}\cap M_Q,A)^{r}$ is the set of reduced roots of $A$ in $N_{P}\cap M_Q$. 
If more generally $P$ is only supposed to be a semi-standard parabolic subgroup of $Q$, we extend the definition by choosing an $s\in W^Q$ so that $P_0=w_sPw_s^{-1}$ is standard and 
we define $\deg_{x}^{Q}(P)$ by requiring that 
\begin{equation} \deg_{x}^{Q}(P)= \deg_{w_s x}^Q(P_0).  \end{equation}
We have (see \cite[Proposition 1.9]{Behrend})   
  \begin{equation}\label{weight}\sum_{\alpha\in \Phi(N_{P_0}\cap M_Q,A)^{r} } \alpha=\sum_{\varpi\in \hat{\Delta}_{B}^{Q}-\hat{\Delta}_{B}^{P_0}}n_{\varpi} \varpi , \end{equation}
with $n_{\varpi}\geq 2$.  As the restriction of an element in $ \hat{\Delta}_{B}^{Q}-\hat{\Delta}_{B}^{P_0} $ to $\ago_B^{P_0}$ is trivial, 
it follows that for any standard parabolic subgroup $P_0$ and $x\in G(\AAA)$, we have
 \begin{equation} \deg_{\delta x}^Q(P_0)=\deg_{x}^Q(P_0), \quad\forall \delta\in P_0(F), \end{equation}
and in particular the definition of $\deg_{x}(P)$ for a semi-standard parabolic subgroup $P$ is independent of the choice of $s$.

Given any complementary polyhedron for a root system, Behrend (\cite[Definition 3.1]{Behrend}) has associated each facet with a degree. Specialised to our case, for any $P\in \mathcal{P}^{Q}(T)$, the above defined  $\deg_x^{Q}(P)$ coincides with Behrend's degree for the facet corresponding to $P\cap M_Q$ with respect to the following complementary polyhedron (Proposition \ref{complementary} and Remark \ref{Q})
\[ (s^{-1}H_B(w_s x)+ s^{-1}[s_{w_s x}\xi] -s^{-1}X ) _{s\in W^Q}, \]
in $\ago_B\cong \ago_{B\cap M_Q}$ for the reduced root system $\Phi(M_Q,A)^{r}$. 
Behrend has proved (\cite[Corollary 3.14, Corollary 3.16]{Behrend}) that for any complementary polyhedron, there is a unique facet (equivalently a semi-standard parabolic subgroup) which is the smallest for the partial order given by inclusions of the closure of the facets (equivalently the largest semi-standard parabolic subgroup) in the set of facets  with maximal degree, this unique facet (parabolic subgroup) is called special.

Fixing an $x\in G(\AAA)$, the set 
\begin{equation}\label{65} \{\deg_{\delta x}^Q(P)| B\subseteq P\subseteq Q, \delta\in P(F)\backslash Q(F)\eta \} \end{equation} has an upper bound (because of equality (\ref{weight}) and \cite[(5.2) p.936]{Arthur} see also \cite[3.5.4]{LabWal}). By discreteness of degree, the upper bound can be achieved, say by a pair $(P_1, \delta_1)$. We suppose that $(P_1, \delta_1)$ is also a largest element with such property for the partial order defined by inclusion. Then the pair $(P_1, \delta_1)$ sharing these properties is unique by Behrend's uniqueness theorem (\cite[Corollary 3.14, Corollary 3.16]{Behrend}) on the special parabolic subgroup for the family of the complementary polyhedra \begin{equation} \label{cp1} (s^{-1}H_B(w_s \delta x)+ s^{-1}[s_{w_s\delta x}\xi] -s^{-1}X ) _{s\in W^Q},  \end{equation}
when varying $\delta \in P_1(F)\delta_1$. 
In fact, for any two parabolic subgroups $(P_1,\delta_1)$ $(P_2, \delta_2)$ contained in $(Q,\eta)$, with $P_1, P_2$ standard, $\delta_1\in P_1(F)\backslash G(F)$ and $\delta_2\in P_2(F)\backslash G(F)$, we must have $\delta_1\delta_2^{-1}\in Q(F)$. Applying  Bruhat decomposition of $Q(F)$ to $\delta_1\delta_2^{-1}$, we can choose as representatives  $\delta_1, \delta_2\in G(F)$ such that $\delta_2=w\delta_1$ for some Weyl element $w$ of $M_Q$.  We see that $\deg_{\delta_1 x}^Q(P_1)$ and $\deg_{\delta_2 x}^Q(P_2)$ (defined by (\ref{degree})) are respectively the degree of the parabolic subgroups $P_1\cap M_Q$ and $wP_2w^{-1}\cap M_Q$ for the complementary polyhedron \[ (s^{-1}H_B(w_s \delta_1x)+ s^{-1}[s_{w_s\delta_1x}\xi ]-s^{-1}X ) _{s\in W^Q}. \]

%Moreover the maximality of $\deg^Q_{\delta_1 x}(P_1)$ is equivalent to that of $\deg^{G}_{\delta_1 x}(P_1)$ when the pair $(P_1, \delta_1)$ stays in the set \eqref{65}, %because the projection of $H_B(\delta x)$ and $[s_{\delta x} \xi ]$ in $\ago_Q$ depends only in $(Q,\eta)$. 

Finally, observe that the statement of the theorem is a reformulation of the \cite[3.10]{Behrend} and the choice of $(P_1, \delta_1)$:  a pair $(P_1, \delta_1)$ shares these properties if and only if it is a $(\xi, X)$-canonical refinement of $(Q,\eta)$. In fact \cite[3.2(iii)]{Behrend} and \cite[3.10.B2]{Behrend} imply the Definition \ref{sscr}.(1.). Let's show that  \cite[3.10.B1]{Behrend} is equivalent to the inequality in Definition \ref{sscr} (2.). We denote the vector $H_B(\delta_1 x)+ [s_{\delta_1 x}\xi]-X$ by $H$ for convenience. Behrend's set $\mathrm{vert}(P_1)$ is $\hat{\Delta}_{P_1}^Q$ in our language. Let $\lambda=\hat{\varpi}_{\alpha}\in \hat{\Delta}_{P_1}^Q$, with $\alpha\in  \Delta_B^Q- \Delta_B^{P_1}$, we need to calculate $n(P_1,\lambda)$. Recall that $n(P_1, \lambda)$ is defined in the definition \cite[3.7]{Behrend} by 
\[ n(P_1,\lambda) = \langle \sum_{\beta \in \Psi(P_1,\lambda) } \beta ,    H  \rangle,
  \]
 while \cite[Lemma 3.6]{Behrend} says that $\sum_{\beta \in \Psi(P_1,\lambda) } \beta \in \mathbb{R}\hat{\Delta}_{P_1}^Q$, hence we have 
 \[ n(P_1,\lambda)=\langle \sum_{\beta \in \Psi(P_1,\lambda) } \beta ,    H_{P_1}  \rangle,  \]
where $H_{P_1}$ is the projection of $H$ in $\ago_{P_1}$.  
Note that by definition \cite[3.5]{Behrend}, an element in $\Psi(P_1, \lambda)$ is of the form $\alpha+\mu$ with $\mu\in \ago_B^{P_1, *}$, therefore 
 \[ n(P_1,\lambda)=  |\Psi(P_1, \lambda) | \langle \alpha ,    H_{P_1}  \rangle.    \]
Thus $n(P_1,\lambda)>0$ is equivalent to the inequality in Definition \ref{sscr} (2.) for the element in $\Delta_{P_1}^Q$ represented by $\alpha$. 
\end{proof}

\subsection{Functions $F^{Q, \xi, X}(\cdot)$}

%\begin{coro}
%Let $g\in G(\AAA)$, $(Q_0, \delta)$ a parabolic subgroup. Let $(P_0, \delta')\subseteq (Q_0, \delta)$ be a parabolic subgroup which is $({\xi}, T)$-semi-stable. 
%\end{coro}

\begin{definition}
Let $Q$ be a standard parabolic subgroup. Let $F^{Q,\xi, X}(x)$ be the function
defined for
$x\in Q(F)\backslash G(\AAA)$ by 
\[ F^{Q,\xi, X}(x)=\sum_{\{P|  B\subseteq P\subseteq Q\}}(-1)^{\dim\ago_P^Q}  \sum_{\delta\in P(F)\backslash Q(F)} \hat{\tau}_P^{Q}(H_B(\delta x) + [s_{\delta x} \xi] - X ). \]
By  Lemma 5.1. of \cite{A1}, the sum is always finite,  hence $F^{Q,\xi, X}(x)$ is well defined.
\end{definition}

\begin{lemm}
For any $x\in G(\AAA)$, $(\xi, X)\in \ago_{B,\infty}\times \ago_{B}$, and $Q$ a standard parabolic subgroup, one has
\begin{equation}\label{Lan}
1=\sum_{\{P|B\subseteq P\subseteq Q\}}\sum_{\delta\in P(F)\backslash Q(F)}{\tau}_P^Q(H_B(\delta x)+  [s_{\delta x}\xi] -X)F^{P, \xi, X}(\delta x).\end{equation}
\end{lemm}
\begin{proof}
We can insert directly the definition of $F^{P, \xi, X}(\delta x)$ into the right-hand-side of the equation. After changing  order of summation, the identity follows by the fact that the matrix $((-1)^{\dim \ago_P}\tau_{P}^{Q}(H))$ (indexed by standard parabolic subgroups $P\subseteq Q$) is the inverse of  
 $((-1)^{\dim \ago_P}\htau_{P}^{Q}(H))$ (Proposition 1.7.2 of \cite{LabWal}). 
\end{proof}

\begin{prop}\label{compact}
For any admissible pair $(\xi, X)\in \ago_{B,\infty}\times \ago_{B}$, 
the function $F^{G, \xi, X}$ is the characteristic function of the set consisting of those $x \in G(\AAA)$ which are $(\xi, X)$-semi-stable, i.e. the set of $x$ such that for any proper standard parabolic subgroup $P$ of $G$, any $\delta\in P(F)\backslash G(F)$, we have \[   \htau_P(H_B(\delta x)+ [s_{\delta x}\xi] -X)=0. \]
In particular the function $F^{G, \xi, X}$ is compactly supported over $G(F)\backslash  G(\AAA)/Z_G(\AAA)$ when $(\xi, X)$ is admissible. 
\end{prop}
\begin{proof}
The proof is similar to that of  \cite[Lemme 2.4.2]{Chau}. 

Note that $F^{G, \xi, X}(x)$ equals 
\[  \sum_{\{(Q, \eta)\}}   (-1)^{\dim\ago_{Q}^G}  \hat{\tau}_Q(H_B(\eta x)+ [s_{\eta x}\xi] -X ), \]
where the sum is taken over the set of all parabolic subgroups of $G$ defined over $F$.  
By existence and uniqueness of $(\xi, X)$-canonical refinement (Theorem \ref{Behrend}), we can organize the above sum by grouping together those $(Q,\eta)$ under which $x$ has the same $(\xi, X)$-canonical refinement. Hence $F^{G, \xi, X}(x)$ equals 
\begin{equation}\label{semi-stable}
\sum_{\{(P, \delta)\}} \sum_{\{(Q,\eta)|\text{cr}(x)^{(Q,\eta)}=(P,\delta)  \}}   (-1)^{\dim\ago_{Q}}  \hat{\tau}_Q(H_B(\eta x) + [s_{\eta x}\xi ] -X ), \end{equation}
where the first sum is taken over $(\xi, X)$-semi-stable parabolic subgroups for $x$, and
the second sum is taken over the set of parabolic subgroups $(Q,\eta)$ such that the $(\xi, X)$-canonical refinement of $(Q,\eta)$ for $x$ is $(P, \delta)$. 

%Note that for any parabolic subgroup $(Q,\eta)$, $\eta$ is only defined modulo $Q(F)$ from left multiplication. 
The inclusion $\delta^{-1}P\delta\subseteq \eta^{-1}Q\eta$ holds if and only if $P\subseteq Q$ and the image of $\delta$ in $Q(F)\backslash G(F)$ is $\eta$. 
Therefore, given a parabolic subgroup $(P,\delta)$ which is $(\xi, X)$-semi-stable for $x$ and a parabolic subgroup $(Q, \eta)$ containing $(P, \delta)$, $(P,\delta)$ is the $(\xi, X)$-canonical refinement of $(Q, \eta)$ for $x$ if and only if  \[ \tau_P^Q( H_B(\delta x)+ [s_{\delta x}\xi] -X)=1.\] Thus the inner sum of the expression (\ref{semi-stable}) equals 
\[ \sum_{\{Q|P\subseteq Q\}}  (-1)^{\dim\ago_{Q}}  \hat{\tau}_Q(H_B(\delta x) + [s_{\delta x}\xi] -X ) {\tau}_P^{Q}(H_B(\delta x)+ [s_{\delta x}\xi] -X ), \]
which is zero except if $P=G$ (\cite[Proposition 1.7.2]{LabWal}). This latter holds if and only if $x$ is $(\xi, X)$-semi-stable.
We conclude that the sum (\ref{semi-stable}) is $1$ if $x$ is $(\xi, X)$-semi-stable and is $0$ otherwise. 

Finally, the assertion about compact support can be easily deduced from the definition of $(\xi, X)$-semi-stability and \cite[Proposition 3.5.3]{LabWal} (applied for $Q=G$). 
\end{proof}

\begin{prop}\label{FP=FM}
Let $Q$ be a standard parabolic subgroup of $G$ and $x\in G(\AAA)$. 
Let $x=nmk$ be an Iwasawa decomposition of $x\in G(\AAA)$ with $n\in N_Q(\AAA)$, $m\in M_Q(\AAA)$ and $k\in G(\mathcal{O})$ such that  $s_{k}\in W$ has minimal length among all such Iwasawa decompositions of $x$. 
By identifying $\ago_B$ with $\ago_{B\cap M_Q}$, one has \[ F^{Q, \xi,X}(x)=F^{M_Q, s_k{\xi}, X}(m). \]
\end{prop}
\begin{proof}
As $P(F)\backslash Q(F)$ is in bijection with $(P\cap M_Q)(F)\backslash M_Q(F)$, and the map $P\mapsto P\cap M_Q$ defines a bijection between the set $\{P|  B\subseteq P\subseteq Q\}$ and the set of standard parabolic subgroups of $M_Q$ (with the fixed Borel subgroup being $B\cap M_Q$). 
By identifying $\ago_B$ with $\ago_{B\cap M_Q}$, we have
\begin{align*}F^{Q, \xi, X}(x)&=\sum_{\{P|  B\subseteq P\subseteq Q\}}(-1)^{\dim\ago^Q_P}  \sum_{\delta\in P(F)\backslash Q(F)} \hat{\tau}_P^{Q}(H_B(\delta x)+ [s_{\delta x}\xi] -X )\\
&= \sum_{\{R\in \mathcal{P}^{M_Q}(B\cap M_Q)\}}(-1)^{\dim\ago_P^{M_Q}}  \sum_{\delta\in R(F)\backslash M_Q(F)} \hat{\tau}^{M_Q}_R(H_B(\delta m)+ [s_{\delta x}\xi] -X  ).
\end{align*}

For any $R\in \mathcal{P}^{M_Q}(B\cap M_Q)$ and $\delta\in M_Q(F)$, let $\delta m=b_0k_0$ with $b_0\in (B\cap M_Q)(\AAA)$ and $k_0\in M_Q(\mathcal{O})$. Then $\delta x= \delta n \delta^{-1} b_0k_0k$. As $\delta n \delta^{-1} b_0\in B(\AAA)$, we have $s_{\delta x}=s_{k_0k}$ and $s_{\delta m}=s_{k_0}$. Note that $s_{k}$ has minimal length in $W^{Q}s_k$, so $s_{k_0k}=s_{k_0}s_k$ by Corollaire 1, Chapter IV $§2$ of \cite{Bourbaki} and Lemme 1.3.3 of \cite{LabWal}. By definition, we obtain $F^{Q,\xi,X}(x)=F^{M_P, s_k\xi, X}(m).$
\end{proof}

%\begin{prop}
%For any $X, Y, Z, \xi\in \ago_B$ with $d(Z)\geq 0$,  the set of $x\in G(\AAA)$ such that there exists $\delta \in G(F)$ 
%$$   \Gamma_Q'(H_B(\delta x)-X, Y)F^{Q}_{{\xi}}(\delta x, Z) \neq 0 $$ is compact modolo $G(F)\Xi_G$. 
%\end{prop}
%\begin{proof}
%Note that if $\Gamma_Q'(H_B( x)-X, Y)\neq 0$, then $x$ belongs to a finite union $$\bigcup_{i} H_Q^{-1}(H_i +\ago_G)  $$ where $H_i\in \ago_Q$ are a finite number of vectors in $\ago_Q$. 

%For any $H\in \ago_Q$.  Let $x\in H_Q^{-1}(H)$.  Suppose $x=nmk$ for $n\in N_Q(A)$ and $m\in M_Q(\AAA)$ and $k\in G(\mathcal{O})$. Then by lemma \ref{FP=FM}, $F^{Q}_{{\xi}}( m, Z) \neq 0$ if and only if $F^{M_Q}_{\underline{s_k}{\xi}}(x, Z)\neq 0$. As  $x\in H_Q^{-1}(H) $ is equivalent to $H_{M_Q}(m)= H$. Let $S^{H}$ be the intersection of the support of $F^{M_Q}_{\underline{s_k}{\xi}}(\cdot , Z)$ and the set of $m$ such that $H_{M_Q}(m)=H$. Then $M_Q(F)\backslash S^{H}$ is compact, hence $G(F)\backslash G(F)N_{Q}(\AAA)S^{H}G(\mathcal{O})$ is compact. 
%\end{proof}

\section{A relation between Jordan-Chevalley decomposition and Levi decomposition}
The hypothesis \((\ast)_G\) or \((\ast)_\ggg\) in \ref{chara} is assumed to hold following the group case or the Lie algebra case respectively. 

\subsection{}
The following lemma is well known and will be used without further mention in this section.  
\begin{lemm}
Let $P$ be a parabolic subgroup of $G$. Let $\gamma\in P(F)$ be a semi-simple element, then $N_{P,\gamma}$ is connected. 
\end{lemm}
\begin{proof}
We can suppose that we're over an algebraic closed field, and suppose $\gamma$ lies in a maximal torus $S$ of $P$, then $N_{P,\gamma}$ is generated by the connected  groups $U_{\alpha}$ for all $\alpha\in \Phi(N_P,S)_{\bar{F}}$ such that $\alpha(\gamma)=1$. 
\end{proof}

The number field case of the following proposition is due to J. Arthur (\cite[Lemma 3.1]{Ageom}). Our statement and strategy of proof is adapted from \cite[Lemma 2.3.]{ChauLie} where a Lie algebra version is proved. Note that our lemma \ref{missing} is used implicitly in the proof of \cite[Lemma 2.3.]{ChauLie}. 
Over a function field, one needs to be more careful with unipotent elements, unipotent subgroups and smoothness properties. 
The third case of the following result is not needed in this article, but will be used in \cite{Yu}. We include it here as it does not increase too much the length of the proof. 
\begin{prop}[Arthur, Chaudouard]\label{unipotent}
Suppose that $P$ is a standard parabolic subgroup of $G$ with standard Levi subgroup $M$ and unipotent radical $N$. Let $\sigma\in M(F)$. 
Suppose that $\sigma$ admits Jordan-Chevalley decomposition with semi-simple part $\sigma_s$ and unipotent part $\sigma_u$.

Let $A$ be one of the following cases: \\
(1) $A=F$;\\
 (2) $A=\mathbb{A}$; \\
(3) In case that $\sigma_s\in M(\mathbb{F}_q)$, we also allow
$A=\mathcal{O}$, the ring of integral adèles. 

 Fixing a system of representatives $\Delta$ for $N_{\sigma_s}(A)\backslash N(A)$ (note that in the third case $N_{\sigma_s}$ is defined over $\mathbb{F}_q$, hence $N_{\sigma_s}(A)$ makes sense),
then for any $x\in N(A)$, there is a unique pair $(n, u)\in \Delta\times N_{\sigma_s}(A)$ such that 
 \[ x=\sigma^{-1}n^{-1}\sigma u n .   \]

In the Lie algebra case, suppose that $\sigma\in \mmm(F)$. Let $A$ be either the case (1) or (2) above or $A=\mathcal{O}$ if $\sigma_s\in \mmm(\mathbb{F}_q)$. 
Fixing a system of representatives $\Delta$ for $N_{\sigma_s}(A)\backslash N(A)$,
then for any $x\in \nnn(A)$, there is a unique pair $(n, u)\in \Delta\times \nnn_{\sigma_s}(A)$ such that 
 \[ x= \Ad(n^{-1})(\sigma+u)-\sigma.   \]

\end{prop}
\begin{proof}
Note that if the lemma is true for one $\sigma$, then it's true for all elements in $M(F)$-conjugacy class of $\sigma$. Hence we can freely take an $M(F)$-conjugate of $\sigma$ if necessary. 

After Proposition \ref{Gille}, we may suppose that $\sigma_u$ belongs to the unipotent radical of a parabolic subgroup $P'$ of $G$, hence it lies in $N_{P'}(F)\cap M(F)$. Then $\sigma_u$ is contained in the unipotent radical of a parabolic subgroup $P_1=P'\cap M$ of $M$. We may suppose that  $P_1$ is a minimal parabolic subgroup of $M$ defined over $F$.  
On the other hand, there is a parabolic subgroup $Q$ of $M$ being minimal with the property that $\sigma_s\in L(F)$ for a Levi subgroup $L$ of $Q$. After conjugation, we suppose that $Q$ contains $P_1$. 

We claim that $\sigma_u$ belongs to the unipotent radical of $Q$. This is because $Q^{0}_{\sigma_s}$ is a parabolic subgroup of $M^{0}_{\sigma_s}$ with a Levi factor $M_{Q,\sigma_s}^{0}$ and with unipotent radical $N_{Q,\sigma_s}$.  As $\sigma_u\in Q^{0}_{\sigma_s}(F)$, its projection to $M_{Q, \sigma_s}^{0}(F)$ belongs to the unipotent radical of a parabolic subgroup of $M_{Q, \sigma_s}^{0}$ (cf. 3.2.Proposition (A) of \cite{Tits2}). While the derived group of $M_{Q, \sigma_s}^{0}$ is anisotropic, hence no proper $F$-parabolic subgroup lies in it. Therefore the projection of $\sigma_u$ in $M_{Q, \sigma_s}^{0}$ is  trivial. 
It implies that $\sigma_u$ belongs to the unipotent radical of $Q^{0}_{\sigma_s}$ and hence to that of $Q$.

Let $R=QN$ (which is a parabolic subgroup of $G$, see 4.4 and 4.7 of \cite{BTu}), we have $\sigma\in R(F)$, $\sigma_u\in N_{R}(F)$ and $R\subseteq P$. Let $A_R$ be the maximal split torus in the center of $M_R$. Then the Lie algebra of $N$ can be decomposed into eigenspaces under the action
of $A_R$. 
Let $\lambda\in X_*(A_R)$ be a cocharacter so that $R=P_G(\lambda):= \{g\in G\mid \lim_{t\rightarrow 0} \lambda(t)g\lambda(t)^{-1} \text{ exists}\}$. For each $i\in \mathbb{Z}$,  let $A_i\subseteq X^*(A_R)$ be the semi-group consisting of elements $\alpha\in X^*(A_R)$ such that $\langle\alpha, \lambda\rangle\geq i$. Applying \cite[3.3.6, 3.3.5]{CGP} to the semi-groups $A_i$, we deduce that $N_{R}$ admits a descending filtration $N_{R}=\tilde{N}_{1}\supseteq \tilde{N}_2\supseteq \cdots\supseteq \tilde{N}_r=\{1\}$ by $R$-conjugate stable smooth closed $F$-subgroups such that: 
\begin{center}
$uu'u^{-1}u'^{-1}\in \tilde{N}_{i+j}(F)$ for $u_i\in \tilde{N}_i(F)$ and $u_j\in \tilde{N}_j(F)$ (one sets $\tilde{N}_{k}=\{1\}$ for $k\geq r$). 
\end{center}
As $N=N_P\subseteq N_R$ which is $A_R$-stable, we obtain a descending filtration \[ N=N_{1}\supseteq N_2\supseteq \cdots\supseteq N_r=\{1\}, \] with $N_i=\tilde{N}_i\cap N$. 
\begin{lemm}\label{filtration}
The filtration constructed above  satisfies the following properties:
\begin{enumerate}
\item
each $N_i$ is a normal subgroup of $N$ and is $\sigma$ conjugate stable ;
\item
 $\sigma_u^{-1}n_i^{-1}\sigma_{u}n_i\in N_{i+1}(A)$ for any $n_i\in N_{i}(A)$. 
 \item
each $N_i$ is a connected smooth closed subgroup of $N$.
\end{enumerate}
\end{lemm}
\begin{proof}[Proof of the lemma]
(1) Because $\tilde{N}_i$ is  $R$-conjugate stable;  (2) Since $\sigma_u\in \tilde{N}_1(F)$.
(3) By Proposition 3.3.9 of \cite{CGP}. 
\end{proof}

Fix a system of representatives $\Delta_k$ for $N_{\sigma_s}(A) N_k(A)\backslash N(A)$ for each$k=1, 2, \cdots$.  
Given $x\in N(A)$, we proceed  by recurrence on $k$ for the following assertion: there exists a unique couple $(n,u)\in \Delta_k\times N_{\sigma_s}(A)N_{k}(A)$ such that \[ x=\sigma^{-1}n^{-1}\sigma u n. \]

The case that $k=1$ is trivial. Suppose that the assertion is proved for $k$, i.e., there exists a unique pair $(n_k, u_k) \in \Delta_k\times N_{\sigma_s}(A)N_{k}(A)$ such that $x=\sigma^{-1} n^{-1}_k\sigma u_k n_k$.
Let's prove the case for $k+1$. 

For unicity, let $(n,u)\in \Delta_{k+1}\times N_{\sigma_s}(A)N_{k+1}(A)$ be a couple such that $x=\sigma^{-1}n^{-1}\sigma un$. 
Let $\beta\in \Delta_{k}$ be the representative of the coset $ N_{\sigma_s}(A)N_{k}(A) n$, let $\alpha:=n\beta^{-1}\in N_{\sigma_s}(A)N_{k}(A) $. Then $n$ is uniquely determined by $\beta$ and the coset $N_{\sigma_s}(A)N_{k+1}(A)\alpha$, i.e. $\{n\}=\Delta_{k+1}\cap ( N_{\sigma_s}(A)N_{k+1}(A)\alpha\beta)$, and $u$ will be automatically unique (assuming its existence). 
%Given any system of representatives $\Delta'_k$ of  $N_{\sigma_s}(A) N_k(A)\backslash N(A)$.The element $n\in \Delta_{k+1}$ can be written as $n=\alpha'\beta'$ for a unique couple $(\alpha', \beta')\in N_{\sigma_s}(A)N_{k}(A)\times \Delta'_{k}$. As the inclusion map $N_k(A)\rightarrow N_{\sigma_s}(A)N_{k}(A)$ induces a surjection $$ N_{k}(A)\longrightarrow   N_{\sigma_s}(A)N_{k}(A)/N_{\sigma_s}(A)N_{k+1}(A),$$
%$\alpha'$ can be written as $\alpha'=\alpha a$ for some  $\alpha\in N_{k}(A)$ and $a\in N_{\sigma_s}(A)N_{k+1}(A)$. Now $\Delta_k:=a\Delta'_k$ is another system of representatives of $N_{\sigma_s}(A) N_k(A)\backslash N(A)$, with which we have a unique decomposition $n=\alpha\beta$ for $\alpha\in N_k(A)$ and $\beta$ ($=a\beta'$) is in $ \Delta_{k}$. 

We have 
\[ x=\sigma^{-1}n^{-1}\sigma u n= \sigma^{-1}   \beta^{-1}  \sigma   (\sigma^{-1} \alpha^{-1}\sigma)  ( u\alpha) \beta. \]
By our construction of $N_i$, we have $\sigma^{-1} \alpha^{-1}\sigma \in N_{\sigma_s}(A)N_k(A)$ and $u\alpha \in N_{\sigma_s}(A)N_{k}(A)$, hence $ (\sigma^{-1} \alpha^{-1}\sigma)  ( u\alpha) \in N_{\sigma_s}(A)N_{k}(A)$. 
By unicity of $(n_k, u_k)$, we have 
\[ n_k=\beta, \]
\[  u_k= (\sigma^{-1} \alpha^{-1}\sigma)  ( u\alpha).  \]
Since $\sigma^{-1}\alpha^{-1}\sigma = (\sigma_s^{-1}(\sigma_{u}^{-1} \alpha^{-1} \sigma_u \alpha)\sigma_s) (\sigma_s^{-1}\alpha^{-1}\sigma_s)$, and both $\sigma_s^{-1}(\sigma_{u}^{-1} \alpha^{-1} \sigma_u \alpha)\sigma_s$  and $\alpha^{-1}u\alpha$ belong to $N_{\sigma_s}(A)N_{k+1}(A) $, 
we have \begin{equation}\label{alpha} N_{\sigma_s}(A)N_{k+1}(A)u_k = N_{\sigma_s}(A)N_{k+1}(A) \sigma_s^{-1}\alpha^{-1}\sigma_s \alpha.  \end{equation}
We will show that this relation determines the coset $    N_{\sigma_s}(A)N_{k+1}(A)  \alpha$. 

\begin{lemm}\label{missing}
Let $A$ be one of the case in Proposition \ref{unipotent}, then
we have an isomorphism of abstract groups: 
 \[ N_{\sigma_s}(A)N_{k+1}(A) \backslash N_{\sigma_s}(A) N_{k}(A)  \cong   
 (N_{k,\sigma_s}N_{k+1}\backslash N_{k})(A)
. \]
\end{lemm}
\begin{proof}[Proof of the lemma]
We have an isomorphism of  groups: 
 \[ N_{\sigma_s}(A)N_{k+1}(A) \backslash N_{\sigma_s}(A) N_{k}(A)  \cong   N_{k, \sigma_s}(A) N_{k+1}(A) \backslash N_{k}(A). \]
By definition, $N_{k,\sigma_s}\cap N_{k+1}=N_{\sigma_s}\cap N_{k+1}=N_{k+1, \sigma_s}$.  
Due to Proposition 3.3.10 of  \cite{CGP}, $N_{k+1, \sigma_s}$ is connected and smooth. Moreover, as it admits an action by a split torus with no non-zero weight, it is $F$-split (Lemma 3.3.8 of $loc.$ $cit.$), 
i.e. it admits a composition series over $F$ whose successive quotients are $F$-isomorphic to $\mathbb{G}_a$. Then one knows (by induction on the length of the composition series) that $H^{1}(F, N_{k+1, \sigma_s})=0$ and $H^{1}(F_v, N_{k+1, \sigma_s})=0$ for any place $v$ of $F$. Since $N_k$ and $N_{k,\sigma_s}$ are smooth,  by short exact sequence
$1\rightarrow N_{k+1, \sigma_s}\rightarrow N_{k,\sigma_s}\ltimes N_{k+1}\rightarrow N_{k,\sigma_s}N_{k+1}\rightarrow 1$ and the vanishing of  $H^1$, we have (\cite[2.3.6]{Yu})
\begin{equation}
N_{k,\sigma_s}(F)N_{k+1}(F)\cong   N_{k+1, \sigma_s}(F) \backslash( N_{k,\sigma_s}(F) \ltimes N_{k+1}(F))   \cong (N_{k,\sigma_s}N_{k+1})(F),\end{equation} and 
\begin{equation}\label{Fv}
N_{k,\sigma_s}(F_v)N_{k+1}(F_v)\cong(N_{k,\sigma_s}N_{k+1})(F_v), \end{equation} for any place $v$ of $F$.  Again as $N_{k,\sigma_s}N_{k+1}$ is connected, smooth, unipotent and $F$-split  (as it's a quotient of $N_{k,\sigma_s}\ltimes N_{k+1}$), we deduce that 
\begin{equation}
(N_{k,\sigma_s}N_{k+1})(F) \backslash N_{k}(F) \cong (N_{k,\sigma_s}N_{k+1}\backslash N_{k})(F),\end{equation}
and 
\begin{equation}(N_{k,\sigma_s}N_{k+1})(F_v) \backslash N_{k}(F_v) \cong (N_{k,\sigma_s}N_{k+1}\backslash N_{k})(F_v) , \end{equation}
for any place $v$ of $F$. So the lemma is proved when $A=F$. 

Consider the case that $A=\mathbb{A}$. The algebraic group $N_{k+1, \sigma_s}$ is smooth and connected, hence geometrically integral, we deduce by I. 3.6 of \cite{Oes}, that the image of the group $N_{k,\sigma_s}(\AAA)N_{k+1}(\AAA)$ is open in $(N_{k,\sigma_s}N_{k+1})(\AAA)$. While by relation (\ref{Fv}), the image of $N_{k,\sigma_s}(\AAA)N_{k+1}(\AAA)$ is also dense in $(N_{k,\sigma_s}N_{k+1})(\AAA)$. 
It follows that \[ N_{k,\sigma_s}(\AAA)N_{k+1}(\AAA)\cong (N_{k,\sigma_s}N_{k+1})(\AAA). \]
Similarly, \[ (N_{k,\sigma_s}N_{k+1})(F_v) \backslash N_{k}(F_v) \cong (N_{k,\sigma_s}N_{k+1}\backslash N_{k})(F_v). \]
This finishes the proof of this case. 

When $\sigma_s\in M(\mathbb{F}_q)$, let's treat the case that $A=\mathcal{O}$. Now the exact sequence $1\rightarrow N_{k+1, \sigma_s}\rightarrow N_{k,\sigma_s}\ltimes N_{k+1}\rightarrow N_{k,\sigma_s}N_{k+1}\rightarrow 1$
is defined over $\mathbb{F}_q$. By above arguments or by Lang's theorem on vanishing of $H^{1}$ for connected algebraic group over finite fields, the homomorphism 
$N_{k,\sigma_s}(\kappa_v ) \ltimes N_{k+1}(\kappa_v ) \rightarrow (N_{k,\sigma_s}N_{k+1})(\kappa_v )$
is surjective for any place $v$ of $F$. Moreover, as the morphism $N_{k,\sigma_s}\ltimes N_{k+1}\rightarrow N_{k,\sigma_s}N_{k+1}$ is smooth, in particular formally smooth, the homomorphism $N_{k, \sigma_s}(\mathcal{O}_v) \ltimes N_{k+1}(\mathcal{O}_v) \rightarrow (N_{k,\sigma_s}N_{k+1})(\mathcal{O}_v)$ is surjective too. Hence 
$N_{k, \sigma_s}(\mathcal{O}_v) N_{k+1}(\mathcal{O}_v) \cong (N_{k,\sigma_s}N_{k+1})(\mathcal{O}_v)$, for any place $v$ of $F$. We obtain \[ N_{k, \sigma_s}(\mathcal{O}) N_{k+1}(\mathcal{O}) \cong (N_{k,\sigma_s}N_{k+1})(\mathcal{O}). \]
Similarly, \[ (N_{k,\sigma_s}N_{k+1})(\mathcal{O}) \backslash N_{k}(\mathcal{O}) \cong (N_{k,\sigma_s}N_{k+1}\backslash N_{k})(\mathcal{O}).  \]
This finishes the case $A=\mathcal{O}$. \end{proof}
Now we're going to finish the proof of Proposition \ref{unipotent}.

Consider the morphism of schemes defined by: 
\begin{align*}
\Phi_{k}:  N_{k, \sigma_s} N_{k+1}\backslash N_{k} &\longrightarrow N_{k,\sigma_s}N_{k+1} \backslash N_{k}\\ 
y&\mapsto (\sigma_s^{-1} y^{-1} \sigma_s)\cdot y
\end{align*}
This is a morphism of algebraic groups as $N_{k, \sigma_s} N_{k+1}\backslash N_{k} $ is commutative. 
Let's prove that $\Phi_k$ is an isomorphism. As $N_{k, \sigma_s} N_{k+1}\backslash N_{k} $ is smooth and connected, it's sufficient to show that $\Phi_k$ induces an injection on geometric points and that the associated map of Lie algebra $\mathrm{Lie}(\Phi_k)$ is surjective.  

If $y\in (N_{k+1} \backslash N_k)(\bar{F})$ is an element such that \[ b:=(\sigma_s^{-1} y^{-1} \sigma_s)\cdot y\in (N_{k+1}\backslash N_{k+1}N_{k, \sigma_s})(\bar{F}).\]
Then $y^{-1}\sigma_s y=\sigma_s b$ is the Jordan-Chevalley decomposition of $y^{-1}\sigma_s y$  (as an element in $(N_{k+1}\backslash R)(\bar{F})$). We must have $b=1$ and $y\in (N_{k+1}\backslash N_{k+1}N_{k, \sigma_s})(\bar{F})$, which proves injectivity on geometric points. 

Now $\Phi_k$ is induced by passing to quotient of the following morphism:
\begin{align*}
\tilde{\Phi}_{k}:  N_{k+1}\backslash N_{k} &\longrightarrow  N_{k+1} \backslash N_{k}   ;\\ 
y&\longmapsto (\sigma_s^{-1} y^{-1} \sigma_s)\cdot y .
\end{align*}
We know that $\mathrm{Lie}(\tilde{\Phi}_k)= \mathrm{Id}-\mathrm{Ad}(\sigma_s)$. Let $p: \mathrm{Lie}(N_{k+1}\backslash N_k) \rightarrow \mathrm{Lie}(N_{k+1}N_{k,\sigma_s}\backslash N_k)$ be the projection, then $\mathrm{ker}(p)=\mathrm{Lie}(N_{k+1}\backslash N_{k+1}N_{k, \sigma_s}) $.
While we also have $\mathrm{ker}(\mathrm{Lie}(\tilde{\Phi}_k))\cong \mathrm{Lie}({\mathrm{ker}(\tilde{\Phi}_k) )}=\mathrm{Lie}(N_{k+1}\backslash N_{k+1}N_{k, \sigma_s})=\mathrm{ker}(p)$. 
Since $\mathrm{Lie}(\tilde{\Phi}_k)$ is a semi-simple endomorphism, 
\[ \mathrm{Lie}(N_{k+1}\backslash N_k)  =\mathrm{ker}(\mathrm{Lie}(\tilde{\Phi}_k))\oplus\mathrm{Im}(\mathrm{Lie}(\tilde{\Phi}_k)). \]
It follows that the composition $p\circ \mathrm{Lie}(\tilde{\Phi}_k)$ is surjective. We conclude then that $\mathrm{Lie}(\Phi_k)$ is surjective.

Now $\Phi_k$ is known to be an isomorphism, combining with above lemma, we conclude that the coset $N_{\sigma_s}(A)N_{k+1}(A) \sigma_s^{-1}\alpha^{-1}\sigma_s \alpha$ is uniquely determined by equality (\ref{alpha}). This finishes the proof of the uniqueness part of the induction.

For existence part, we see by the above argument that there is an $\alpha\in N_{\sigma_s}(A)N_k(A)$ such that the equality (\ref{alpha}) is satisfied. Let $\beta=n_k$,  $n$ be the unique element in $\Delta_{k+1}\cap  ( N_{\sigma_s}(A)N_{k+1}(A)\alpha\beta)$ and $u= \sigma^{-1} n\sigma xn^{-1}$. Using induction hypothesis, we have $u=\sigma^{-1} nn_k^{-1} \sigma u_k n_k n^{-1}$, we see that $u\in N_{\sigma_s}(A)N_{k}(A)$.  
\end{proof}

\begin{coro}\label{descent}
An element $p\in P(F)$  (resp. $\ppp(F)$),  with $P$ a standard parabolic subgroup of $G$, admits Jordan-Chevalley decomposition if and only if its Levi factor does, and in this case their semi-simple parts are $N_P(F)$-conjugate.   

In particular, let $o\in \mathcal{E}$, then for any pair of standard parabolic subgroup $P\subseteq Q$, we have 
\[ o\cap M_Q(F)\cap P(F) = (o\cap M_P(F)) N_P^Q(F), \]
in the group case
and 
\[ o\cap \mmm_Q(F)\cap \ppp(F) = (o\cap \mmm_P(F)) + \nnn_P^Q(F), \]
in the Lie algebra case.  
\end{coro}
\begin{proof}
We prove the group case as the Lie algebra case can be proved by the same arguments. 

Let $p\in P(F)$, and $p=\sigma x$ be the Levi decomposition of $p$ with $\sigma\in M_P(F)$ and $x\in N_P(F)$. It amounts to prove that $\sigma$ lies in the equivalence class of $p$ in $\mathcal{E}$. 

If $\sigma$ admits Jordan-Chevalley decomposition with semi-simple part $\sigma_s$, then by Proposition \ref{unipotent}, we can write $p=n^{-1}\sigma u n$ with $u\in N_{P,\sigma_s}(F)$ and $n\in N_{P}(F)$. Therefore $p$ must also have Jordan-decomposition with semi-simple part $n^{-1}\sigma_s n$ and we're done. 

If $\sigma$ does not admit Jordan-Chevalley decomposition, we need to prove that nor does $p$. But this is clear since  $M_P$ is a Levi subgroup of $P$ means that $M_P\cong P/N_P$.  If $p$ has a Jordan-Chevalley decomposition, then its image under the quotient map gives a Jordan-Chevalley decomposition of $\sigma$. 
\end{proof}

\section{Proof of the Theorem \ref{Main}  and the Theorem \ref{polynomial}} \label{proofm}
\subsection{Proof of the Theorem \ref{Main}}
The proof follows basically from the same strategy of \cite[Proposition 11, p.227]{Laf}. 
We need a lemma first. 
\begin{lemm}\label{invariance}
Let $f\in {C}_c^{\infty}(G(\AAA))$ or $f\in {C}_c^{\infty}(\ggg(\AAA))$ depending on the case we're dealing with. There is a constant $c$ depending  on $f$ and $\xi$, such that: for a couple of standard parabolic subgroups  $P\subseteq Q$, a vector $X\in \ago_B$ with $d(X)> c$, then for any element $x\in G(\mathbb{A})$ satisfying  \[ F^{P,\xi, X}(x)\tau_{P}^{Q}(H_B(x)+[s_{x}\xi]-X)=1, \]
we have 
%$$\sum_{\gamma\in M_P(F)\cap o}  \int_{N_P(\AAA)} f(x^{-1}\gamma nx)\d n=\sum_{\gamma\in M_Q(F)\cap o}  \int_{ N_Q(\AAA)} f(x^{-1}\gamma nx)\d n ,$$
\[ k_{P, o}(x)=k_{Q,o}(x), \]
or for the Lie algebra case  \[ \kkk_{P, o}(x)=\kkk_{Q,o}(x).  \]
\end{lemm}
\begin{proof}
We prove the group case as the Lie algebra case is similar.  

%Note that the statement is unchanged if we multiply an element in $P(F)$ from left to $x$. Modifying $x$ if necessary, after Iwasawa decomposition, we can write $x=anmk$, with $a\in Z_M(\AAA)$, $n\in N_P(\AAA)$, $k\in G(\ooo)$, and $m$ lies in a fixed compact subset of  $ M_P(\AAA)$ (since $F^{P,\xi, X}(x)=1$).  

Note that we have 
\[ \sum_{\gamma\in M_Q(F)\cap o}  \int_{ N_Q(\AAA)} f(x^{-1}\gamma nx)\d n =\int_{N_Q(F)\backslash N_Q(\AAA)}  \sum_{\gamma\in (M_Q(F)\cap o)N_Q(F) }   f(x^{-1}\gamma nx)\d n.  \]
After modifying $x$ from left by an element in $P(F)$, which doesn't affect the statement of the lemma, we may write using Iwasawa decomposition that $x=bak$ with $b\in N_B(\AAA)$ lying in a fixed compact subset, $a\in T(\AAA)$, and $k\in G(\ooo)$. 
The condition \[ F^{P,\xi, X}(x)\tau_{P}^{Q}(H_B(x)+[s_{x}\xi]-X)=1\] implies that (\cite[1.2.7]{LabWal})
\begin{equation} \label{ineq1}  \alpha(H_B(a)-X)>c', \quad \forall \alpha\in \Delta_B^Q -\Delta_B^{P},\end{equation}
where the constant $c'$ depends only $\xi$. 
Taking a $X_G\in \ago_B$ such that $-X_G$ is deep enough in the positive chamber (depending on the choice of a Siegel domain), then we may suppose, after changing $x$ from left by an element in $P(F)$ if necessary,  that $\alpha(H_B(a)-X_G)>0$, for all $\alpha\in \Delta_B^P$. Since $\alpha(X-X_G)>0$ for all $\alpha\in \Delta_B$, we have 
\begin{equation} \label{ineq2}  \alpha(H_B(a)-X_G)>0, \quad \forall \alpha\in \Delta_B^Q. \end{equation}

If $\gamma\in Q(F)$ satisfies that $x^{-1}\gamma nx$ lies in the support of $f$ for some $n\in N_Q(\AAA)$ lying in a fixed fundamental domain of $N_Q(F)\backslash N_Q(\AAA)$, then $a^{-1}b^{-1}\gamma b a $ lies in a compact subset depending only in the support of $f$. Hence when $d(X)$ is large enough, we must have $\gamma\in P(F)$ (cf. \cite[3.6.6]{LabWal}). 
Note that by Corollary \ref{descent}, 
\[  \left((M_Q(F)\cap o)N_Q(F) \right)\cap P(F)=(M_P(F) \cap o)N_P(F).    \]

It remains to show that, if $X$ is deep enough in the positive chamber, then 
\[ \int_{N_P(F)\backslash N_P(\AAA)} \sum_{\gamma\in (M_P(F)\cap o)N_P(F) } f(x^{-1}\gamma nx) \d n=\int_{N_Q(F)\backslash N_Q(\AAA)}\sum_{\gamma\in (M_P(F)\cap o)N_P(F) }   f(x^{-1}\gamma nx)\d n. \]
This follows by a similar proof of \cite[I.2.8 Lemma]{MW} (there is a tiny omission in their proof as $V''\subseteq V$ can not be satisfied by only requiring $a^{\alpha}>c$). 
For reader's convenience, we complete the proof.

Let $N_P^Q=M_Q\cap N_P$. Then we have a decomposition $N_P(\AAA)=N_Q(\AAA) N_P^Q(\AAA)$.
Note that we can moreover modify $x$ from left by an element in $N_Q(\AAA)$, hence using the above decomposition of $x$, we can write $x=bmak$ with $b\in N_P^Q(\AAA)$ and $m\in M_P(\AAA)$ are in  fixed compact subsets, $a\in T(\AAA)$ and $k\in G(\ooo)$.
 
 We can also decompose the integral over $N_P(F)\backslash N_P(\AAA)$ into a double integral over \[N_P^Q(F)\backslash N_P^Q(\AAA)\times N_Q(F)\backslash N_Q(\AAA).\] For $n\in N_P(\AAA)$, let $n=n_Qn_P^Q$ be the associated decomposition. We have \[f(x^{-1}\gamma n x)=f((x^{-1}\gamma n_Q x ) (x^{-1} n_P^Q x) ).\] Since $f$ is smooth with compact support, it is invariant from right by a compact open subgroup $K$. Let $K'$ be a compact open subgroup such that $k^{-1}K'k\subseteq K$ for any $k\in G(\ooo)$.  Choose a compact open subset $V$ of $N_P^Q(\AAA)$, such that  
the projection $V\rightarrow N_P^Q(F)\backslash N_P^Q(\AAA)$ is bijective, then the integral over $N_P^Q(F)\backslash N_P^Q(\AAA)$ is the same thing as the integral over $V$. If we were in the local field case, then when $X$ is deep enough in the positive chamber, $a^{-1}m^{-1} b^{-1}Vbma$ would be contained in $K'$. In our global case, whenever $X$ is deep enough in the positive chamber, we conclude from the local case and weak approximation theorem that there is an $a_f\in T(F)$ such that, $a_f^{-1}a^{-1} m^{-1}b^{-1}Vbmaa_f$ is contained in $K'$. However as $x$ can be modified from left by an element in $P(F)$ without changing the statement, we can replace $a$ by $a_fa$. 
\end{proof}

Now we come back to the proof of the Theorem \ref{Main}. We only prove the group case since the Lie algebra case can be proved by similar arguments.

For any $\xi\in \ago_{B,\infty}$, let $X'\in \ago_{B}$ be a vector such that the pair $(\xi, X')$ is admissible and $d(X')$ is larger than the constant given in the lemma \ref{invariance}. 
Insert the identity (\ref{Lan}) into the expression of $k^{\xi, X}_{o}(x)$, after simplification, it equals the sum over all pairs of standard parabolic subgroups $P\subseteq Q$, followed by the sum over $\delta\in P(F) \backslash G(F)$ (where we have combined the sum over $Q(F)\backslash G(F)$ with the sum over $P(F)\backslash Q(F)$) of the following expression 
\begin{equation}
 (-1)^{\dim \ago_Q^G}\hat{\tau}_Q(H_B(\delta x)+[s_{\delta x}\xi]- X )  {\tau}_P^Q(H_B(\delta x) + [s_{\delta x}\xi] -X' ) F^{P, \xi, X'}(\delta x) k_{Q,o}(\delta x). \end{equation}

Note that Lemma \ref{invariance} says that the expression
\[ {\tau}_P^Q(H_B(\delta x) + [s_{\delta x}\xi] -X' ) F^{P, \xi, X'}(\delta x) k_{Q,o}(\delta x) \]
equals 
 \[ {\tau}_P^Q(H_B(\delta x) + [s_{\delta x}\xi] -X' ) F^{P, \xi, X'}(\delta x) k_{P,o}(\delta x).\]
Hence, we can arrange the order of the sum in $k_o^{\xi,X}$ over the set of pairs of standard parabolic subgroups $P\subseteq Q$. We deduce that $k_o^{\xi,X}$ equals the sum over standard parabolic subgroups $P$ followed by the sum over $\delta\in P(F)\backslash G(F)$ of the expression $F^{P, \xi, X'}(\delta x) k_{P,o}(\delta x)$ times 
\begin{equation}\label{oki}
\sum_{\{Q|P\subseteq Q\}}(-1)^{\dim \ago_Q^G}\hat{\tau}_Q(H_B(\delta x)+ [s_{\delta x}\xi]-X ){\tau}_P^Q(H_B(\delta x)+[s_{\delta x}\xi] -X' ).\end{equation}
We denote the sum (\ref{oki}) by $\Gamma_{P}(H_B({\delta x})+s_{\delta x}\xi-X' ,X-X')$, where \[\Gamma_{P}(H,X)=\sum_{\{Q|Q\supseteq P\}} (-1)^{\dim\ago_Q^G}\tau_P^Q(H)\htau_Q(H-X) \] is a function on $\ago_B\times \ago_B$. One can find more details about it in \cite[Section 1.8]{LabWal}. For us, we only need to know that 
\begin{enumerate}
\item 
For any $X\in \ago_B$, the projection of the support of the function $H\mapsto \Gamma_P(H,X)$ in $\ago_P^G$ is compact (\cite[1.8.3]{LabWal}).
\item
For any $H\in \ago_B$, $\Gamma_{P}(H, 0)=0$ if $P\neq G$ and  $\Gamma_{G}(H, 0)=1$ (\cite[1.8.1, 1.8.3]{LabWal}).
\end{enumerate}

In summary, for $X'$ deep enough, $k^{\xi, X}_o(x)$ equals 
\begin{equation}\label{anotherex}
\sum_{P\in \mathcal{P}(B)} \sum_{\delta\in P(F)\backslash G(F)} \Gamma_{P}(H_B({\delta x})+ [s_{\delta x}\xi]-X' ,X-X') F^{P, \xi, X'}(\delta x) k_{P,o}(\delta x). 
\end{equation}
We know that the function \[ x\mapsto \Gamma_{P}(H_B({ x})+[s_{ x}\xi]-X' ,X-X') F^{P,\xi, X'}(x) \] has compact support 
in $P(F)\backslash G(\AAA)/Z_G(\AAA)$ thanks to our Proposition \ref{FP=FM} and Proposition \ref{compact}. Therefore, we conclude that $k_o^{\xi,X}(x)$ always have compact support in $G(F)\backslash G(\AAA)/Z_G(\AAA)$. 
 
When $d(X)\gg 0$, one can take simply $X'=X$. Since $\Gamma_{P}(H, 0)=0$ for any $H$ if $P\neq G$, we conclude that   \[ k_o^{\xi,X}(x)= F^{G,\xi, X}(x)\sum_{a\in \Xi_G}\sum_{\gamma \in o}f(ax^{-1}\gamma x). \]

\subsection{Proof of the Theorem \ref{polynomial}}
We need a lemma first. Suppose $M$ is a standard Levi subgroup. Let's choose a lattice $\Xi_M$ in $Z_M(F)\backslash Z_M(\AAA)$ containing $\Xi_G$. 
The following lemma is a variant of \cite[4.5.5]{Chau}. 
\begin{lemm}[Chaudouard]\label{quasi-p} \sloppy
For any standard parabolic subgroup $P$ with Levi subgroup $M=M_P$, any $H_0 \in \ago_{P}$, the function $X\in \ago_{B, \mathbb{Q}}\mapsto \sum_{H\in H_P(\Xi_M)/H_P(\Xi_G)  } \Gamma_Q(H+H_0,X)$ is a quasi-polynomial. 
\end{lemm}
\begin{proof}
The proof is essentially the same as that of \cite[4.5.5]{Chau}. 

Let $L_0\subseteq \ago_{B,\mathbb{Q}}$ be any lattice. 
We can fix an integer $N$ that is divisible enough, such that for any parabolic subgroup $P\subseteq Q$, we have
\[ H_P(\Xi_M)   \subseteq \frac{1}{N}  \mathbb{Z}(\hat{\Delta}_P^{Q, \vee}) \oplus \frac{1}{N}  \mathbb{Z}(\Delta_Q^{\vee})   \oplus \frac{1}{N}  \ago_{G, \mathbb{Z}},  \]
and 
\[ L_0  \subseteq \frac{1}{N}  \mathbb{Z}(\hat{\Delta}_P^{Q, \vee}) \oplus \frac{1}{N}  \mathbb{Z}(\Delta_Q^{\vee})   \oplus \frac{1}{N}  \ago_{G, \mathbb{Z}}. \]

For any $\lambda \in \ago_{P,\mathbb{C}}^{*}$, let \[ f_{\lambda}(X):=\sum_{H\in H_P(\Xi_M)/H_P(\Xi_G)  }  \Gamma_P(H+H_0,X)q^{-\langle\lambda, H\rangle}. \]
We need to prove that $f_{0}(X)$ is a quasi-polynomial in $X$. As the sum above is finite (the support of $\Gamma_P(\cdot, X)$ is compact in $\ago_P^G$), for any $X\in \ago_B$, $f_{\lambda}(X)$ is a holomorphic function in $\lambda$. While by definition of $\Gamma_P$, we have \[ f_\lambda(X)=\sum_{\{Q\mid P\subseteq Q\}} (-1)^{\dim \ago_Q^G} f_{\lambda}^{Q}(X), \]
where $f_\lambda^{Q}(X)$ is the series defined by
 \[ f_\lambda^{Q}(X)=\sum_{H\in H_P(\Xi_M)/H_P(\Xi_G)  }   \tau^{P}_{Q}(H+H_0)\htau_{P}(H+H_0-X)q^{-\langle\lambda, H\rangle}. \]

Let $L^{*}_Q\subseteq \ago_{P}^{*}$ be the lattice dual to $L_Q= \frac{1}{N}  \mathbb{Z}(\hat{\Delta}_P^{Q, \vee}) \oplus \frac{1}{N}  \mathbb{Z}(\Delta_Q^{\vee})   \oplus \frac{1}{N}  \ago_{G, \mathbb{Z}} $ and $H_P(\Xi_M)^{*} \subseteq \ago_{P}^{*}$ the lattice dual to $H_P(\Xi_M)$. Note that $L_Q\supseteq H_P(\Xi_M)$. Let $\mathfrak{f}$ be a system of representatives for the quotient $\frac{2\pi i}{\log q} H_P(\Xi_M)^{*}/ \frac{2\pi i}{\log q} L^{*}_Q $. Then using Fourier transform on the finite abelian group $L_Q/H_P(\Xi_M)$, we have
\[ f_\lambda^{Q}(X) = |\frac{ H_P(\Xi_M)^{*} }{L_Q^{*} }  |^{-1}\sum_{\nu\in \mathfrak{f}} \sum_{H\in L_Q/H_P(\Xi_G)  }   \tau^{Q}_{P}(H+H_0)\htau_{Q}(H+H_0-X)q^{-\langle\lambda+\nu, H\rangle}  . \]
After a direct calculation, for any $X\in L_0$, $f_\lambda^{Q}(X) $ is equal to $ |\frac{ H_P(\Xi_M)^{*} }{L^{*}_Q }  |^{-1} |\frac{\frac{1}{N}  \ago_{G,\mathbb{Z}}      }{H_P(\Xi_G) }  |$ times
\begin{multline*}      \sum_{\nu\in \mathfrak{f}}   \prod_{{\varpi}^{\vee}_{\alpha}      \in    \hat{\Delta}_P^{Q, \vee} }  
\frac{q^{ - \langle \lambda+\nu,    \varpi^{\vee}_{\alpha}  \rangle ( [-N\langle\alpha, H_0\rangle]  +1) /N }  }{1- q^{-\langle \lambda+\nu,    \varpi^{\vee}_{\alpha}      \rangle /N } }
  \prod_{\beta^{\vee}\in \Delta^{\vee}_Q }   \frac{  q^{ -  \langle \lambda+\nu,   \beta^{\vee}  \rangle ( N\langle\varpi_{\beta},   X \rangle+  [N\langle\varpi_{\beta}, -H_0   \rangle]  +1)/N  }             }{ 1- q^{-\langle \lambda+\nu,    \beta^{\vee}      \rangle/N }       } ,
  \end{multline*}
where we use $[x]$ to denote the largest integer smaller than or equal to $x$. 
  
As $f_{\lambda}(X)=\sum_{\{Q | P\subseteq Q  \}}(-1)^{\dim \ago_P^G} f_{\lambda}^{Q}(X)$ is holomorphic on $\ago_{P,\mathbb{C}}^{*}$, for $\lambda=0$, $f_0(X)$ equals the alternative sum of constant terms of Laurent expansions of $f_{\lambda}^{Q}(X)$ around $0$, which shows that the result holds. 
\end{proof}

Now we come back to the proof of the Theorem \ref{polynomial}.

Using the expression (\ref{anotherex}) of $k^{\xi, X}_{o}(x)$, the integral $J^{\xi,X}_o$ equals the sum over  standard parabolic subgroups $P\in \mathcal{P}(B) $ of %the integral over $x\in P(F)\backslash G(\AAA)/\Xi_G$ of the following expression$$\Gamma_{P}(H_B({x})+ s_{x}\xi-X', X-X')F^{P,\xi,X'}( x) k_{P,o}(x).  $$
%$$\sum_{ P\in \mathcal{P}(B)  }\int_{P(F)\backslash G(\AAA)/\Xi_G}  \Gamma_{P}(H_B({x})+ s_{x}\xi-X', X-X')F^{P,\xi,X'}( x) k_{P,o}(x) \d x.  $$
\begin{equation}\label{xj}
\int_{P(F)\backslash G(\AAA)/\Xi_G} \Gamma_{P}(H_B({x})+ [s_{x}\xi ] -X', X-X')F^{P,\xi,X'}( x) k_{P,o}(x) \d x  .   \end{equation}

Let's prove that for each standard parabolic subgroup $P$, the expression (\ref{xj}) is a quasi-polynomial in $X$. Using Iwasawa decomposition, we can decompose the integral into a double integral:
\[ \int_{P(F)\backslash P(\AAA)/\Xi_G} \int_{G(\ooo)} \Gamma_{P}(H_B(p)+[ s_{k}\xi]-X', X-X')F^{P,\xi,X'}(pk) k_{P,o}(pk) \d k \d p. \]
By definition $ s_{k} $ depends on the coset in $G(\ooo)/\mathcal{I}_{\infty}\prod_{v\neq \infty}G(\ooo_v)$ represented by $k$, and by Proposition \ref{FP=FM}, the same is true for $F^{P,\xi,X'}(pk)$. Moreover, note that the function $p\mapsto k_{P,o}(pk)$ is invariant under multiplication by $Z_{M_P}(\AAA)$, in particular by $\Xi_M$. The same is clearly true for $p\mapsto F^{P,\xi,X'}(pk)$ too. 
Thus, the integral (\ref{xj}) equals the integral over ${u\in G(\ooo_\infty)/\mathcal{I}_\infty}$ followed by the integral over $p\in {P(F)\backslash P(\AAA)/\Xi_M } $ of the product of \[ F^{P,\xi,X'}(pu)    \int_{\mathcal{I}_{\infty}\prod_{v\neq \infty}G(\ooo_v) }  k_{P,o}(puk) \d k,  \] with 
\begin{equation} 
 \sum_{H \in H_P(\Xi_M)/H_P(\Xi_G)  }    \Gamma_{P}(H+ H_B(p)+[ s_{u}\xi]-X', X-X')    .   \end{equation}
By lemma \ref{quasi-p}, this last expression is a quasi-polynomial in $X$, hence $X\mapsto J^{\xi, X}$ is a quasi-polynomial.

\section{Another expression for $J^{\xi, X}_{o}$} \label{essunip}
%Assume first that the class $o\in \mathcal{E}$ admits Jordan-Chevalley decomposition. 
Let $Q$ be a standard parabolic subgroup of $G$, set 
\begin{equation}
j_{Q,o}(x)=\sum_{a\in \Xi_{G}} \sum_{\gamma\in M_Q(F)\cap o} \sum_{n\in N_P(F)}f(ax^{-1}\gamma n x).  
\end{equation}
%Note that by Proposition \ref{unipotent}, 
%\begin{equation} j_{Q, o}(x) =\sum_{a\in \Xi_{G}} \sum_{\gamma \in M_Q(F)\cap o} \sum_{\nu \in N_{Q, \gamma_s}(F) \backslash N_Q(F) } \sum_{n\in N_{Q, \gamma_s}(F)} f(ax^{-1} \nu^{-1} \gamma n \nu x)  , \end{equation}
%where $\gamma=\gamma_s\gamma_u$ is the Jordan-Chevalley decomposition in $G(F)$ of $\gamma$ with $\gamma_s$ its semi-simple part. 
We can eaisly verify that $j_{Q,o}(x)$ is both $M_Q(F)$ and $N_Q(F)$ left invariant, hence $Q(F)$ left invariant. 
Therefore, for any $x\in G(\AAA)$, we can define a function $j^{\xi, X}_{o}(x)$ using the same expression (\ref{kQo}) as $k^{\xi, X}_{o}(x)$ (page \pageref{kQo}) but replacing $k_{Q,o}$ by $j_{Q,o}$. 

In the Lie algebra case, when $G=GL_n$ and $f$ is some special test function, 
Chaudouard has showed (cf. \cite[3.9-3.11]{Chau}) that the integral of $j^{\ggg\lll_n,0,0}$  (the sum of $j^{\ggg\lll_n,0,0}_o$ over all $o\in \mathcal{E}$) over $G(F)\backslash G(\AAA)/\Xi_G$ is closely related to the groupoid volume of semi-stable Higgs bundles. Moreover, he has showed in \cite[Corollaire 5.2.2]{Chau} that the integral  of $j^{\ggg\lll_n,0,0}$ coincides with $J^{\ggg\lll_n,0,0}$. Inspired by his results, we have the following proposition. 
\begin{prop}\label{redu}
Let $(\xi, X)\in \ago_{B,\infty}\times \ago_{B, \mathbb{Q}}$ be any pair of vectors.  
The function $j^{\xi, X}_{o}$ has compact support in $G(F)\backslash G(\AAA)/Z_G(\AAA)$. Moreover, we have 
\begin{equation} \label{vari}J^{\xi,X}_{o}(f) = \int_{G(F)\backslash G(\AAA)/\Xi_G}  j^{\xi, X}_{o}(x)  \d x.   \end{equation}
\end{prop}
\begin{proof}
As showed by the proof of lemma \ref{invariance}, if $X'$ is deep enough (depending on the support of $f$) in the positive chamber, then for any pair of standard parabolic subgroups $P\subseteq Q$ and $x\in G(\AAA)$, 
\[  {\tau}_P^Q(H_B(x) + [s_{ x}\xi ]-X' ) F^{P, \xi, X'}( x) j_{Q,o}( x) \]
equals 
 \[ {\tau}_P^Q(H_B( x) + [s_{ x}\xi] -X' ) F^{P, \xi, X'}( x) j_{P,o}( x).\]
Therefore, the proof of the Theorem \ref{Main}  can be applied word by word here after replacing $k_{Q,o}(x)$ by $j_{Q,o}(x)$, which shows that $j^{\xi, X}_{o}$ has compact support in $G(F)\backslash G(\AAA)/Z_G(\AAA)$ and when $X$ is deep enough in the positive chamber 
\[ j_o^{\xi,X}(x)= F^{G,\xi, X}(x)\sum_{a\in \Xi_G}\sum_{\gamma \in o}f(ax^{-1}\gamma x), \quad \forall x\in G(\AAA). \]
Hence $j_o^{\xi,X}(x)$ coincides with $k_{o}^{\xi, X}(x)$ if $X$ is deep enough. 
Moreover, the proof of Theorem \ref{polynomial} can also be applied word by word after replacing $k_{Q,o}(x)$ by $j_{Q,o}(x)$, which shows that $X\in \ago_{B, \mathbb{Q}}\mapsto \int_{G(F)\backslash G(\AAA)/\Xi_G}  j^{\xi, X}_{o}(x)  \d x$ is a quasi-polynomial. It coincides with $J^{\xi,X}_{o}(f)$ when $X$ is deep enough in the positive chamber, hence coincides with $J^{\xi,X}_{o}(f)$ for every $X\in \ago_{B,\mathbb{Q}}$. 

Another way to prove the last statement of the theorem is by checking that the constant term of $j_{Q,o}$ along $Q$ is $k_{Q,o}$,  which we  leave to the reader. 
%Let's consider the integral  $$    \int_{N_Q(F)\backslash N_Q(\AAA)} j_{Q, o}(\nu x) \d \nu.   $$ It equals  $$\sum_{\gamma\in M_Q(F)\cap o}  \int_{N_{Q, \gamma_s}(F)\backslash N_Q(\AAA)} \int_{N_{Q, \gamma_s}(\AAA)}    f(x^{-1}\nu^{-1}\gamma n \nu x)     \d n \d \nu. $$ We can break the integral over $N_{Q, \gamma_s}(F)\backslash N_Q(\AAA)$ into a double integral over $  N_{Q,\gamma_s}(F) \backslash N_{Q,\gamma_s}(\AAA)  \times   N_{Q, \gamma_s}(\AAA)\backslash N_Q(\AAA)$. Then the integral over $N_{Q,\gamma_s}(F) \backslash N_{Q,\gamma_s}(\AAA)$ can be absorbed by the integral over $n\in N_{Q, \gamma_s}$ which is again a corollary of the proposition \ref{unipotent}. We deduce by comparing with (\ref{interm}) that   $$     \int_{N_Q(F)\backslash N_Q(\AAA)} J_{Q, o}(\nu x) \d \nu = k_{Q,o}(x). $$ Therefore, for any standard parabolic subgroup $P$ contained in $Q$, we have  $$     \int_{N_P(F)\backslash N_P(\AAA)} J_{Q, o}(\nu x) \d \nu =  \int_{N_P(F)\backslash N_P(\AAA)}  k_{Q,o}(\nu x)\d \nu. $$ Therefore, following the proof of  theorem \ref{Main}, we clearly have \begin{equation} J^{\xi,X}_{o}(f) = \int_{G(F)\backslash G(\AAA)/\Xi_G}  j^{\xi, X}_{o}(x)  \d x.   \end{equation}
\end{proof}

\section{Some applications: non-existence or existence and uniqueness of some cuspidal automorphic representations} \label{appl}

In this section, we use the coarse expansion of the trace formula in this article to prove some results suggested to me by M. Harris.  

\subsection{A non-existence theorem}\label{10.1}
Now let $F=\mathbb{F}_q(t)$ be the function field of the projective line $\mathbb{P}^1=\mathbb{P}^{1}_{\mathbb{F}_q}$. 
Let $G$ be a split reductive group defined over $\mathbb{F}_q$. Let $\lambda, \mu \in \mathbb{P}^1(\mathbb{F}_q)$ be two distinct $\mathbb{F}_q$-rational points, identified as two places of $F$. 
Let $T_\lambda$ and $T_\mu$ be two maximal torus of $G$ defined over $\mathbb{F}_q$. Suppose that $\theta_\lambda$ and $\theta_\mu$ are characters in general position of $T_\lambda(\mathbb{F}_q)$ and $T_\mu(\mathbb{F}_q)$ respectively. We obtain by Deligne-Lusztig induction (fixing any isomorphism between $\mathbb{C}$ and $\overline{\mathbb{Q}}_\ell$) two irreducible representations 
\[\rho_\lambda=\epsilon_{T_\lambda}\epsilon_G R_{T_\lambda}^{G}(\theta_\lambda) \text{ and }  \rho_\mu=\epsilon_{T_\mu}\epsilon_G R_{T_\mu}^{G}(\theta_\mu)\] of $G(\mathbb{F}_q)$, where $\epsilon_{H}=(-1)^{rk_{\mathbb{F}_q}H}$ is the sign according to the parity of the split rank of $H$ (the rank of a maximal split sub-torus). 
By inflation, we obtain irreducible representations of $G(\mathcal{O}_\lambda)$ and $G(\mathcal{O}_\mu)$ respectively, still denoted by $\rho_\lambda$ and $\rho_\mu$.

In \cite{Yu}, we have defined the notion of cuspidal filter. Let $\rho$ be the representation of $G(\mathcal{O})$ as the tensor product of $\rho_\lambda$, $\rho_\mu$ and the trivial representation of $G(\mathcal{O}_v)$ for $v$ different from $\lambda, \mu$. Then $\rho$ is a cuspidal filter if for any $L^2$-automorphic representation $\pi$ of $G(\AAA)$, the $\rho$-isotypic part $\pi_\rho$ is non zero implies that $\pi$ is cuspidal.

We have given in \cite[Proposition 5.1.2]{Yu} a criterion for $\rho$ to be a cuspidal filter. 
If one of $T_\lambda$, $T_\mu$ is maximal anisotropic then $\rho$ is automatically a cuspidal filter. 
Suppose that $T_\lambda=T$ is split. Let $M_\mu$ be the minimal Levi subgroup containing $T_\mu$ defined over $\mathbb{F}_q$. Without loss of generality, we suppose that $M_\mu$ contains $T$. Then $\rho$ is a cuspidal filter if \begin{equation}\label{11} \theta_\lambda^w \theta_\mu |_{Z_M(\mathbb{F}_q)} \neq 1,  \end{equation}
for any Weyl element $w\in W$ and for any maximal semi-standard Levi subgroup $M$ of $G$ that contains $M_\mu$.

\begin{example}
For $SL_2$, we choose $T$ to be the group of diagonal matrix with determinant $1$, so that  $\mathbb{F}_q^{\times}\cong T(\mathbb{F}_q)$ via the map $x\mapsto (x, x^{-1})$.  Let $T_\lambda=T_\mu=T$,   $\theta_\mu$ and $\theta_\lambda$ be two characters of $\mathbb{F}_q^\times$, then $\rho$ is a cuspidal filter if \[ \theta_{\mu}\neq \theta_\lambda^{\pm 1}. \]
\end{example}

We prove the following theorem using the coarse geometric expansion developed in this article by taking $\xi$ to be in general position and using some results in \cite{Yu}. 
\begin{theorem}\label{nonexistence}
Suppose that $G\neq T$, $T_\lambda=T$ is split and $\rho$ is a cuspidal filter. 

There is no cuspidal automorphic representation $\pi=\otimes'\pi_v$ of $G(\AAA)$ such that $\pi_\lambda$ contains $(G(\mathcal{O}_\lambda), \rho_\lambda)$,  $\pi_\mu$ contains $(G(\mathcal{O}_\mu), \rho_\mu)$ and all other local components are unramified. 
\end{theorem}
\begin{proof}
Let $\mathcal{I}_\lambda$ be the standard Iwahori subgroup of $G(F_\lambda)$ and $\mathcal{O}^{\lambda}=\prod_{v\neq \lambda}\mathcal{O}_v$. We have a canonical morphism $\mathcal{I}_\lambda\rightarrow B(\mathbb{F}_q)$. By inflation, the character $\theta_\lambda$ of $T(\mathbb{F}_q)$ defines a character of $\mathcal{I}_{\lambda}$. 
Let $\rho$ be the representation of $G(\mathcal{O}^{\lambda})\mathcal{I}_\lambda$ defined as the tensor product of $\theta_\lambda$ of $\mathcal{I}_\lambda$, $\rho_\mu$ of $G(\mathcal{O}_\mu)$ and the trivial representation of $G(\mathcal{O}_v)$ for places $v$ outside $S=\{\lambda,\mu\}$.
Let \[e_\rho=\begin{cases} \frac{1}{vol(\mathcal{I}_\lambda)} \mathrm{Tr}(x^{-1}|\rho) , \quad x\in G(\mathcal{O}^{\lambda})\mathcal{I}_\lambda; \\
0, \quad x\notin G(\mathcal{O}^{\lambda})\mathcal{I}_\lambda.
\end{cases}
\]
After \cite[Proposition 5.2.6]{Yu}, we have \[J^{G,\xi}(e_\rho)=\sum_{\pi}m_\pi\dim \mathrm{Hom}_{G(\mathcal{O})}(\rho, \pi), \]
where the sum is taken over the set of isomorphism classes of cuspidal automorphic representations of $G(\AAA)$ and $m_\pi$ is the multiplicity of $\pi$ in the cuspidal automorphic spectrum, and we choose $\lambda$ as the point $\infty$, $\xi$ is chosen to be in general position in the sense that  the projection of $\xi$ to $\ago_P$ does not belong to $X_{*}(P)+\ago_G$ for any semi-standard parabolic subgroup $P\subsetneq G$.

Using the coarse geometric expansion established earlier (Corollary \ref{5.3}), we have 
\[J^{G,\xi}(e_\rho)=\sum_{o\in \mathcal{E}}  J^{G,\xi}_{o}(e_\rho).  \]
We showed in \cite[Theorem 3.2.5]{Yu} that \( J^{G,\xi}_{o}(e_\rho)=0, \) 
if $o$ is not represented by an elliptic element $\sigma\in T(\mathbb{F}_q)$, i.e. $\sigma\in T(\mathbb{F}_q)$ and satisfies $Z_{G_\sigma^0}^0=Z_G^0$. Note that the theorem \textit{loc. cit.} is stated for split semi-simple reductive group but its proof works without semi-simplicity. 
For the situation we are discussing, we want to show further that \[ J^{G,\xi}_{[\sigma]}(e_\rho)=0, \] 
even for an elliptic element $\sigma\in T(\mathbb{F}_q)$, then this completes the proof.

Using \cite[Theorem 3.2.5]{Yu} again, it suffices to prove that 
\(  J^{G^0_\sigma,w\xi}_{[\sigma]}(e^{w^{-1}}_\rho) =0, \)
for any $w\in W^{(G_\sigma^0, T)}\backslash W$ which sends positive roots of $(G_\sigma^0, T)$ to positive roots. 
Note that for such a $w$, the group $w\mathcal{I}_\lambda w^{-1}\cap G_{\sigma}^0(F_\lambda)$ is the standard Iwahori subgroup $\mathcal{I}_{\lambda, \sigma}$ of $G_{\sigma}^0(F_\lambda)$ and $w\xi$ is also in general position. Without loss of generality,  let's prove that \[  J^{G^0_\sigma, \xi}_{[\sigma]}(e_\rho) =0. \]

Let $f=\otimes f_v\in \mathcal{C}_c^{\infty}(G_\sigma^0(\AAA))$ be a function defined as the tensor product of $f_\lambda=\mathbbm{1}_{\mathcal{I}_\lambda}$, \[f_\mu= \begin{cases}R_{T_\mu'}^{G_{\sigma}^0}(1)(x), \quad x\in G(\mathcal{O}_\mu); \\
0, \quad x\notin G(\mathcal{O}_{\mu}).
\end{cases}
 \] where $T_\mu'$ is any maximal torus of $G_\sigma^0$ defined over $\mathbb{F}_q$, and $\mathbbm{1}_{G_\sigma^0(\mathcal{O}_v)}$ for places $v$ outside $\{\lambda, \mu\}$. After the character formula of Deligne-Lusztig (\cite[Theorem 4.2]{DL}) which applies for every semi-simple element $\sigma\in G(\mathbb{F}_q)$ and every unipotent element $u\in G_\sigma^0(\mathbb{F}_q)$:
\[  R_{T_\mu}^{G}(\theta_\mu)(\sigma u) =  \frac{1}{|G_\sigma^{0}(\mathbb{F}_q)|}   \sum_{\{     \gamma \in G(\mathbb{F}_q)\mid \sigma\in \gamma T_\mu  \gamma^{-1}    \}}      \theta_\mu(\gamma^{-1}\sigma\gamma)  R_{\gamma T_\mu\gamma^{-1}}^{G_{\sigma}^0}(1)(u), \]
it suffices to prove that \[J^{G_\sigma^0,\xi}_{unip}(f)=0,\] where $J^{G_\sigma^0,\xi}_{unip}=J^{G_\sigma^0,\xi}_{[1]}$. For this, we're going to pass to Lie algebra. Let $\ggg_\sigma$ be the Lie algebra of $G_\sigma^0$ and $\mathfrak{I}_{\sigma,\lambda}\subseteq\ggg_\sigma(\mathcal{O}_\lambda)$ be the standard Iwahori Lie sub-algebra.

Let $\varphi=\otimes \varphi_v\in \mathcal{C}_c^{\infty}(\ggg_\sigma(\AAA))$, where for $v$ outside $\{\lambda,\mu\}$, 
\[ \varphi_v=\mathbbm{1}_{\ggg_\sigma(\ooo_v)}; \]
for $v=\lambda$, 
\[ \varphi_{\lambda} =  \mathbbm{1}_{ \mathfrak{I}_{\sigma,\lambda}}, \] and for $v=\mu$,  \begin{equation*}
\varphi_{\mu}= q^{\frac{1}{2}l^2+\frac{1}{2}l-1 } \hat{\mathbbm{1}_{\Omega_{- t_\mu}  }},  
\end{equation*}
where $t_\mu$ is a semi-simple regular element in $\ttt'_\mu(\mathbb{F}_q)$, and $\hat{\mathbbm{1}_{\Omega_{- t_\mu}  }}$ is the Fourier transform of the characteristic function of the set $\Omega_{- t_\mu}$ consisting of elements $x\in \ggg_\sigma(\mathcal{O}_\mu)$ whose reduction mod-$\wp_\mu$ lies in the $G(\mathbb{F}_q)$-conjugacy class of $-t_\mu$. The support of $\varphi$ is contained in $\ggg_\sigma(\mathcal{O}^{\lambda})\mathfrak{I}_\lambda$. Therefore we have
\[ J^{\ggg_\sigma, \xi}(\varphi)= |\mathfrak{z}_{\ggg_\sigma}(\mathbb{F}_q)| J^{\ggg_\sigma, \xi}_{nilp}(\varphi). \]
Springer's hypothesis proved by Kazhdan (\cite[Theorem A.1]{KV}) tells us that \[  \varphi_\mu(u) = f_\mu(l(u)) , \] 
for any unipotent element $u\in G_\sigma^0(\mathcal{O}_\mu)$ and any Springer's isomorphism (which exists since we're in very good characteristic) $l: \mathcal{U}_{G_\sigma^0}\xrightarrow{\sim} \mathcal{N}_{\ggg_\sigma}$ defined over $\mathbb{F}_q$.

Let $\langle \cdot, \cdot \rangle$ be a $G$-invariant bilinear form on $\ggg$ defined over $\mathbb{F}_q$. Thus, it defines by taking respectively $\AAA$-points, and $F_v$-points (for every place $v$), a bilinear form on $\ggg(\AAA)$ and $\ggg(F_v)$ respectively. 
For $\mathbb{P}^1=\mathbb{P}^{1}_{\mathbb{F}_q}$, the canonical line bundle is isomorphic to $\mathcal{O}_{\mathbb{P}^{1}}(-2)$ and $K_{\mathbb{P}^{1}}=-\lambda-\mu$ is a canonical divisor. Let $\wp_v$ be the maximal ideal of $\mathcal{O}_v$, let $\wp^{-K_{\mathbb{P}^{1}}}=\wp_\lambda \wp_\mu \prod_{v\neq \lambda , \mu} \mathcal{O}_v$. 
We have 
\[ \mathbb{A}/(F+  \wp^{-K_{\mathbb{P}^{1}}  } )\cong H^{1}(\mathbb{P}^1, \mathcal{O}_{\mathbb{P}^{1}}(-2))\cong H^{0}(\mathbb{P}^1,\mathcal{O}_{\mathbb{P}^1})^{*}\cong\mathbb{F}_{q},\] by Serre duality. 
Fix a non-trivial additive character $\psi$ of $\mathbb{F}_q$. 
Via above isomorphisms, $\psi$ can be viewed as a character of $\AAA/F$. We use this $\psi$ in the definition of Fourier transform. 
By the trace formula of Lie algebra established earlier (Theorem \ref{TFL}), we have 
\[J^{\ggg_\sigma, \xi}(\varphi) = q^{(1-g)\dim \ggg_\sigma} J^{\ggg_\sigma, \xi}(\hat{\varphi}).  \]
The key point is that the Fourier transform $\hat{\varphi}$ has support in $\ggg_\sigma(\mathcal{O}^{\lambda})\mathfrak{I}_\lambda $ as well (\cite[5.3.2]{Yu}). Therefore, by our vanishing theorem \cite[Theorem 3.2.5]{Yu} again
\[ J^{\ggg_\sigma,\xi}(\hat{\varphi})=\sum_{z\in \zzz_{\ggg_\sigma}(\mathbb{F}_q)} J_{[z]}^{\ggg_\sigma,\xi}(\hat{\varphi}). \]
It's clear that for any $z\in \zzz_{\ggg_\sigma}(\mathbb{F}_q)$, we have
 \[ J_{[z]}^{\ggg_\sigma,\xi}(\hat{\varphi})=0 . \]
In fact, the function $\hat{\varphi}_\mu$ is supported in the set of regular semi-simple elements. Therefore \( J_{o}^{\ggg_\sigma,\xi}(\hat{\varphi})=0 ,\) if $o$ is not represented by a regular semi-simple element in $\ggg_\sigma(F)$. However $z$ is never regular if $G_\sigma^0$ is not a torus. Since $\sigma$ is an elliptic element contained in $T(\mathbb{F}_q)$ and the torus $T$ is split, $G_\sigma^0$ is not a torus if $G$ is not a torus.  
\end{proof}

\subsection{$SL_l$ with $l$ being a prime} 
As in Section \ref{10.1}, we still suppose that $F=\mathbb{F}_q(t)$ and $\lambda, \mu$ are two distinct places of $F$ of degree $1$.

Let $l$ be a prime number. 
Let $G=SL_l$. Recall that $p$ is a very good prime for $SL_l$ if $p\neq l$.  Let $T_\lambda=T_\mu$ be an  anisotropic maximal torus defined over $\mathbb{F}_q$. Recall that $\rho_\lambda$ is cuspidal if and only if $T_\lambda$ is anisotropic (since $\theta_\lambda$ is in general position).  
\begin{theorem}\label{102}
 If $\rho_\lambda$ is  contragredient to $\rho_\mu$, then there is exactly one cuspidal automorphic representation $\pi=\otimes'\pi_v$ of $G(\AAA)$ (with cuspidal multiplicity $1$) such that $\pi_\lambda$ contains $(G(\mathcal{O}_\lambda), \rho_\lambda)$,  $\pi_\mu$ contains $(G(\mathcal{O}_\mu), \rho_\mu)$ and all other local components are unramified. 
If $\rho_\lambda$ is not contragredient to $\rho_\mu$, then there is no such cuspidal automorphic representation. 
\end{theorem}
\begin{proof}
In $SL_l$, there is exactly one conjugacy class of anisotropic tori, which are isomorphic to the norm $1$ torus $ \mathrm{R}_{\mathbb{F}_{q^n}| \mathbb{F}_q}^1 \mathbb{G}_{m} $. Hence  $T_\lambda(\mathbb{F}_q)=T_\mu(\mathbb{F}_q)$ is the group of norm $1$ elements in $\mathbb{F}_{q^n}^{\times}$ for the norm map of $\mathbb{F}_{q^n}|\mathbb{F}_q$.  
Since Deligne-Lusztig induction preserves the property of being contragredient (\cite[Proposition 11.4]{DM}), $\rho_\lambda$ is contragredient to $\rho_\mu$ if and only if $(T_\lambda, \theta_\lambda)$ is conjugate to $(T_\mu, \theta_\mu^{-1})$. Using the identification $W\cong \mathfrak{S}_l$, the torus $T_\lambda$ corresponds to the conjugacy class of a cyclic permutation $w$ of length $l$. Moreover, the group $W(G, T_\lambda)(\mathbb{F}_q)$ is isomorphic to the cyclic group generated by $w$. 
We have
\[T_\lambda(\mathbb{F}_q)\cong (\mathbb{F}_{q^l}^{\times})^{\mathrm{Nm}=1}:= \{ s\in \mathbb{F}_{q^l}^{\times} | s^{q^{l-1}+q^{l-2}+\cdots+1   }=1   \} ,   \]
and, under the above isomorphisms, the generator $w$ sends $s\in (\mathbb{F}_{q^l}^{\times})^{\mathrm{Nm}=1}$ to $s^q$. Therefore, $\theta_\lambda$ being in general position implies that \begin{equation} \theta_\lambda^{q^i}\neq \theta_\lambda, \quad \text{  if } l\nmid i.   \end{equation}
If $\rho_\lambda$ is contragredient to $\rho_\mu$, we can assume that $\theta_\lambda=\theta_\mu^{-1}$. If $\rho_\lambda$ is not contragredient to $\rho_\mu$ we have $\theta_\lambda\neq \theta_\mu^{-q^i}$ for all $i\in \mathbb{Z}$, hence
\begin{equation}\label{10.3}\theta_\lambda^{q^i}\theta_\mu^{q^j}\neq 1, \quad \forall i,j\in \mathbb{Z}. \end{equation}
Moreover, we clearly have $\theta_\lambda^{q^i-1}|_{Z_G(\mathbb{F}_q)}=1$ for all $i\in\mathbb{N}$ since the order of $Z_G(\mathbb{F}_q)$ divides $q-1$. 
A final remark is that we may assume \begin{equation}\label{10.4}(\theta_\lambda\theta_\mu)|_{Z_G(\mathbb{F}_q)}=1,\end{equation} otherwise there is no cuspidal automorphic representation with the desired properties as we can see by checking its central character which must be trivial on $Z_G(\mathbb{F}_q)\subseteq Z_G(F)$.

For any automorphic representation $\pi=\otimes'\pi_v$ of $SL_l(\AAA)$. 
Since $\rho_\lambda$ is cuspidal, if $\pi_\lambda$ contains $\rho_\lambda$, then $\pi_\lambda$ is supercuspidal and hence $\pi$ must be cuspidal automorphic. We will do an explicit calculation with our coarse geometric expansion by setting $\xi=0$. 

Let $\rho$ be the representation of $G(\mathcal{O})$ defined as the tensor product of $\rho_\lambda$ (resp. $\rho_\mu$) of $G(\mathcal{O}_\lambda)$ (resp. of $G(\mathcal{O}_\mu)$) and the trivial representation of $G(\mathcal{O}_v)$ for places $v$ outside $S=\{\lambda,\mu\}$. Let \[e_\rho=\begin{cases}  \mathrm{Tr}(x^{-1}|\rho) , \quad x\in G(\mathcal{O}); \\
0, \quad x\notin G(\mathcal{O}).
\end{cases}
\]
After \cite[Proposition 5.2.6]{Yu}, we have
\[ J^{G}(e_\rho)=\sum_{\pi}m_\pi\dim \mathrm{Hom}_{G(\mathcal{O})}(\rho, \pi). \]
It suffices to prove that $J^{G}(e_\rho)$ equals $1$ or $0$ depending on whether or not $\rho_\lambda$ and $\rho_\mu$ are contragredient.

Using our coarse geometric expansion (Corollary \ref{5.3}): 
\[J^{G}(e_\rho)=\sum_{o\in \mathcal{E}}  J^{G}_{o}(e_\rho).  \]
%Applying the theorem \ref{main} for those $X$ such that $d(X)$ large enough and the theorem \ref{polynomial}, we see that \[ J^{G,\xi}_{o}(e_\rho)=0, \]
First of all, suppose $o\in \mathcal{E}$ is a class such that $J_{o}^{G}(e_\rho)\neq 0$.
As $X\mapsto J_{o}^{G, X}(e_\rho)$ is a quasi-polynomial (Theorem \ref{polynomial}), there is an $X\in \ago_B$ which we can assume to be deep enough in the positive chamber, so that $J_{o}^{G, X}(e_\rho)\neq 0$. 
Since the support of $e_\rho$ is contained in $G(\ooo)$, the Theorem \ref{Main} implies that there is an element $\gamma\in o$ and $x\in G(\AAA)$ such that $x^{-1}\gamma x\in G(\ooo)$.  However, this implies that $\gamma^n \in xG(\ooo)x^{-1}\cap G(F)$ for any $n$. Note that $xG(\ooo)x^{-1} \cap G(F)$ is finite, so the element $\gamma$ is a torsion element. By \cite[Proposition 2.3.2]{Yu}, $\gamma$ admits Jordan-Chevalley decomposition. Moreover, \cite[Theorem 2.4.1]{Yu} implies that we can take as representative a semi-simple element $\sigma\in G(\mathbb{F}_q)$. Furthermore, the Deligne-Lusztig induced character $R_{T_\lambda}^{G}(\theta_\lambda)$ is supported in the set of elements whose semi-simple part can be conjugate to an element in $T_\lambda(\mathbb{F}_q)$. Therefore, we conclude that 
\[  J^{G}(e_\rho) = \sum_{s\in T_\lambda(\mathbb{F}_q) /\sim \text{conj}  }  J_{[s]}^{G} (e_\rho).  \]

Note that two elements in $s_1, s_2\in T_\lambda(\mathbb{F}_q)$ are conjugate if there is $1\leq i\leq l$ such that $s_1^{q^{i}}=s_2$. The fact that $l$ is a prime implies that any element in $T_\lambda(\mathbb{F}_q)$ is either regular and elliptic or lies in center of $G=SL_l$. If $s\in T_\lambda(\mathbb{F}_q)$ is regular, then there are $l$ elements in $T_\lambda(\mathbb{F}_q)$ that are conjugate to $s$. Therefore we have  
\begin{equation} J^{G}(e_\rho) = \frac{1}{l} \sum_{s\in T_\lambda(\mathbb{F}_q) - Z_{G}(\mathbb{F}_q) }  J_{[s]}^{G} (e_\rho)   +  |Z_G(\mathbb{F}_q)| J_{unip}^{G} (e_\rho)  .  \end{equation}

Suppose that $s\in T_\lambda(\mathbb{F}_q)$ is regular and elliptic, there is no proper Levi subgroup of $G$ containing $s$. We have by definition
\[ J_{[s]}^{G} (e_\rho) = \int_{G(F)\backslash G(\AAA)}  \sum_{\gamma \in T_\lambda(F) \backslash G(F)} e_{\rho}(x^{-1}\gamma^{-1} s\gamma x )\d x.   \]
Therefore   \[ J_{[s]}^{G} (e_\rho) = \mathrm{vol}(T_\lambda(F)\backslash T_\lambda(\AAA))   O_{s}(e_\rho),  \]
where \[ O_{s}(e_\rho)=\int_{T_{\lambda}(\AAA)\backslash G(\AAA)}e_\rho(x^{-1}sx)\d x. \]
If $x\in G(\AAA)$ is an element such that $x^{-1}sx\in G(\mathcal{O})$, then by \cite[Proposition 7.1]{Ko3}, $x^{-1}sx$ and $s$ are conjugates by an element in $G(\mathcal{O})$. In this case, we have 
\[ e_\rho(x^{-1}sx)= (\sum_{i=1}^{l} \theta_\lambda(s)^{q^i})(\sum_{i=1}^{l}\theta_\mu(s)^{q^i}). \]
If $x^{-1}sx\notin G(\mathcal{O})$, then $e_\rho(x^{-1}sx)=0$. 
We obtain that \[ O_{s}(e_\rho)=(\sum_{i=1}^{l} \theta_\lambda(s)^{q^i})(\sum_{i=1}^{l}\theta_\mu(s)^{q^i})O_{s}(\mathbbm{1}_{G(\mathcal{O})}). \]
To calculate $O_{s}(\mathbbm{1}_{G(\mathcal{O})})$, we can decompose it into product of local orbital integrals: \[ O_{s}(\mathbbm{1}_{G(\mathcal{O})})=\prod_{v} \int_{T_{\lambda}(F_v) \backslash G(F_v)} \mathbbm{1}(x^{-1}sx) \d x.   \]
We use the above result of Kottwitz again, then we have 
\[\int_{T_{\lambda}(F_v) \backslash G(F_v)} \mathbbm{1}(x^{-1}sx) \d x = \frac{\mathrm{vol}(G(\mathcal{O}_v))}{\mathrm{vol}(T_\lambda(\mathcal{O}_v))} .  \]
We obtain that for $s\in T_\lambda(\mathbb{F}_q)-Z_{G}(\mathbb{F}_q)$, 
\begin{equation}
J^{G}_{[s]}(e_\rho)=\frac{\mathrm{vol}(T_\lambda(F)\backslash T_\lambda(\AAA))}{\mathrm{vol}(T_\lambda(\mathcal{O}))}(\sum_{i=1}^{l} \theta_\lambda(s)^{q^i})(\sum_{i=1}^{l}\theta_\mu(s)^{q^i}). 
\end{equation}

For any non-trivial character $\theta$ of $T_\lambda(\mathbb{F}_q)$ that is trivial on $Z_G(\mathbb{F}_q)$ we have, \[ \sum_{s\in T_\lambda(\mathbb{F}_q)-Z_G(\mathbb{F}_q)} \theta(s)=-|Z_G(\mathbb{F}_q)|.  \]
Therefore, if  $\rho_\mu$ is not contragredient to $\rho_\lambda$, then applying \eqref{10.3} and \eqref{10.4}, we have
\[   (\sum_{i=1}^{l} \theta_\lambda(s)^{q^i})(\sum_{i=1}^{l}\theta_\mu(s)^{q^i}) = -|Z_G(\mathbb{F}_q)|  l^{2} .   \]
If $\rho_\mu$ is contragredient to $\rho_\lambda$, then there are $l$ trivial characters among $\theta_\mu^{q^{i}}\theta_\lambda^{q^{j}}$ ($1\leq i ,j\leq l$), we deduce similarly that 
\[   (\sum_{i=1}^{l} \theta_\lambda(s)^{q^i})(\sum_{i=1}^{l}\theta_\mu(s)^{q^i}) = - |Z_G(\mathbb{F}_q)| (l^{2}-l)+   ( \frac{q^l-1}{q-1}-|Z_G(\mathbb{F}_q)|)l.   \]
Here we have used the fact that  $|T_\lambda(\mathbb{F}_q)|= \frac{q^l-1}{q-1}$. 
%\[ T_\lambda(\mathbb{F}_q)\cong (\mathbb{F}_{q^{l}}^{\times})^{\mathrm{Nm}=1},  \]

To calculate  ${\mathrm{vol}(T_\lambda(F)\backslash T_\lambda(\AAA))}/{\mathrm{vol}(T_\lambda(\mathcal{O}))}$, we may use the short exact sequence:
\[1 \longrightarrow  T_\lambda(\mathcal{O})/T_\lambda(\mathbb{F}_q) \longrightarrow  T_\lambda(\mathbb{A})/T_\lambda(F)\longrightarrow T_\lambda(\mathbb{A})/T_\lambda(F)T_\lambda(\mathcal{O})\longrightarrow 1.   \]
It's easy to see that $T_\lambda(\mathbb{A})/T_\lambda(F)T_\lambda(\mathcal{O})$ is trivial since $F$ is the function field of the projective line $\mathbb{P}^{1}_{\mathbb{F}_q}$ which has trivial Jacobian variety. We obtain \begin{equation}\frac{\mathrm{vol}(T_\lambda(F)\backslash T_\lambda(\AAA))}{\mathrm{vol}(T_\lambda(\mathcal{O}))}=\frac{1}{|T_\lambda(\mathbb{F}_q)|}=\frac{q-1}{q^l-1}. \end{equation}

To summarize, we have \begin{equation}\label{10.8}J^{G}(e_\rho)  = \begin{cases}  |Z_G(\mathbb{F}_q)|J^{G}_{unip}(e_\rho) -  (\frac{q-1}{q^l-1}) |Z_G(\mathbb{F}_q)| l  + 1, \quad \text{$\rho_\lambda$   is contragredient to $\rho_\mu$ };
\\   |Z_G(\mathbb{F}_q)|J^{G}_{unip}(e_\rho)- (\frac{q-1}{q^l-1})|Z_G(\mathbb{F}_q)| l, \quad \text{$\rho_\lambda$   is not contragredient to $\rho_\mu$}  .
  \end{cases}  \end{equation}

In the following, we calculate $J^{G}_{unip}(e_\rho)$. We will derive it from calculations of a trace formula of Lie algebra.

As in the proof of the Theorem \ref{nonexistence}, we fix a $G$-equivariant, non-degenerate bilinear form $\langle, \rangle$ defined  over $\mathbb{F}_q$. We also use the same additive character $\psi$ of $\mathbb{A}/F$ for our Fourier transform. 

Let's fix two regular elements $x$ and $y$ in $\mathfrak{t}_\lambda(\mathbb{F}_q)$ which are not $G(\overline{\mathbb{F}}_q)$-conjugate (equivalently not $G({\mathbb{F}}_q)$-conjugate). Let $\Omega_x$ (resp. $\Omega_y$) be the characteristic function over $\mathfrak{g}(\mathbb{F}_q)$ of the conjugacy class of $x$ (resp. of $y$). We have a character defined by
\begin{align*} \chi_\lambda: \mathfrak{t}_\lambda(\mathbb{F}_q)&\longrightarrow \mathbb{C}^{\times} ;\\ z&\longmapsto \psi(\langle z,x\rangle). 
\end{align*}
Similarly we define $\chi_\mu$.

Since the bilinear form $\langle, \rangle$ is non-degenerate on
$\ttt_\lambda(\mathbb{F}_q)$ (\cite[Lemma 5]{McN}), the fact $x$ is not conjugate to $y$ implies that $\chi_\lambda$ is not conjugate to $\chi_\mu$. We know that for the torus $T_\lambda$, the Frobenius element acts by an element of Weyl group, therefore the conjugates of $\chi_\lambda$ (resp. $\chi_\mu$) are $\chi_\lambda^{q^i}$ (resp. $\chi_\mu^{q^{i}}$), $i=1,\ldots, l$. We deduce that $\chi_\lambda^{q^i}\chi_\mu^{q^j}$ is non-trivial for any $1\leq i,j\leq l$. By a result of Letellier \cite[7.3.3]{Letellier} (note that his Fourier transform is normalized differently from the one used here), the Fourier transform of the characteristic function $\mathbbm{1}_{\Omega_x}$ is given by a Lie algebra analogue of the Deligne-Lusztig induced character: for two commuting elements $s,n\in \ggg(\mathbb{F}_q)$ with $s$ semi-simple and $n$ nilpotent: 
 \begin{equation}\label{DLf} \hat{\mathbbm{1}_{\Omega_x}}(s+n)= \frac{1}{|G_s(\mathbb{F}_q)|}   \sum_{\{     \gamma \in G(\mathbb{F}_q)\mid s\in \Ad(\gamma) (T_\lambda )    \}}      \chi_\lambda(\mathrm{Ad}(\gamma^{-1})s)  R_{\gamma T_\lambda \gamma^{-1}}^{G_{s}}(1)(l^{-1}(n)) ,    \end{equation}
where $l$ is any Spinger's isomorphism $l: \mathcal{U}_{G}\rightarrow\mathcal{N}_\ggg$ defined over $\mathbb{F}_q$. 
Let $\varphi=\otimes_v\varphi_v\in C_c^{\infty}(\ggg(\AAA))$ with $\varphi_v=\mathbbm{1}_{G(\mathcal{O}_v)}$ if $v\neq \lambda, \mu $, and $\varphi_{\lambda}$ (resp. $\varphi_\mu$) is a function supported in $\ggg(\mathcal{O}_\lambda)$ (resp. $\ggg(\mathcal{O}_\mu)$) so that for any element $x\in \ggg(\mathcal{O}_\lambda)$, the value $\varphi_\lambda(x)$ (resp. $\varphi_\mu(x)$) is given by the right-hand side of \eqref{DLf} with $s,n$ being respectively the semi-simple part and nilpotent part of $\bar{x}\in \ggg(\mathbb{F}_q)$. Using the quasi-polynomial behaviour (Theorem \ref{polynomial}) of $X\mapsto J^{\ggg,X}_{nilp}(\varphi)$ and $X\mapsto J_{unip}^{G, X}(e_\rho)$ and the Theorem \ref{Main}, we deduce that 
\[ J_{nilp}^{\ggg}(\varphi) =J^{G}_{unip}(e_\rho).\]

We know that the Fourier transform of ${\varphi}$ (by a direct calculation, see \cite[5.3.2]{Yu}) is $\otimes_v\hat{\varphi}_v$ with $\hat{\varphi}_v= \mathbbm{1}_{\ggg(\mathcal{O}_v)}$ for $v\neq \lambda, \mu$, \[ \hat{\varphi}_\lambda(z) =\begin{cases} q^{-\frac{1}{2}(l^2-l)}\mathbbm{1}_{\Omega_x}(-\bar{z}), \quad z\in \ggg(\mathcal{O}_\lambda) ; \\       0, \quad z\notin \ggg(\mathcal{O}_\lambda)  ;   \end{cases} \]
where $\bar{z}$ is the reduction mod $\wp_\lambda$ of $z$, 
and $\hat{\varphi}_\mu$ is given by a similar formula. It's then easy to see that \[ J^{\ggg}(\hat{\varphi})=0. \]
In fact, $\hat{\varphi}$ is supported in $\ggg(\mathcal{O})$, therefore $ J_{o}^{\ggg}(\hat{\varphi})$ vanishes except if $o$ is represented by a semi-simple element in $s\in G(\mathbb{F}_q)$ (\cite[2.4.2]{Yu}). If $J_{[s]}^{\ggg}(\hat{\varphi})\neq 0$, by
 \cite[Proposition 7.1]{Ko3}, we deduce that $s$ is conjugate both to $-x$ and to $-y$. However, by the choices of $x$ and $y$, this is impossible.

By our trace formula of the Lie algebra (Theorem \ref{TFL}), we get that 
\[ J^{\ggg}(\varphi) =0.  \] 
As in the group case, \[ J^{\ggg}_{o}(\varphi)=0,  \]
except if $o$ is represented by an element in $s\in \mathfrak{t}_{\lambda}(\mathbb{F}_q)$.
We obtain 
\[ J_{nilp}^{\ggg}(\varphi) =- \sum_{o\in \mathcal{E}-\{[0]\}}J^{\ggg}_{o}(\varphi)=-\frac{1}{l}\sum_{s\in  \mathfrak{t}_{\lambda}(\mathbb{F}_q)-\{0\}  }  J^{\ggg}_{o}(\varphi)  , \]
where the last equality follows from the fact that every element in $\mathfrak{t}_{\lambda}(\mathbb{F}_q)-\{0\}$ is regular hence is conjugate to $l$ elements in $\mathfrak{t}_{\lambda}(\mathbb{F}_q)-\{0\}$. 
Moreover every element in $\mathfrak{t}_\lambda(\mathbb{F}_q)-\{0\}$ is also elliptic. As in the group case again, we find that for $s\in \mathfrak{t}_\lambda(\mathbb{F}_q)-\{0\}$:
\[J_{[s]}^{\ggg}(\varphi)=\frac{\mathrm{vol}(T_\lambda(F)\backslash T_\lambda(\AAA))}{\mathrm{vol}(T_\lambda(\mathcal{O}))}(\sum_{i=1}^{l} \chi_\lambda(s)^{q^i})(\sum_{i=1}^{l}\chi_\mu(s)^{q^i}). 
\]
Therefore, \[ J^{\ggg}_{nilp}(\varphi)=-\frac{1}{l}\sum_{s\in \mathfrak{t}_\lambda(\mathbb{F}_q)-\{0\}} \frac{\mathrm{vol}(T_\lambda(F)\backslash T_\lambda(\AAA))}{\mathrm{vol}(T_\lambda(\mathcal{O}))}(\sum_{i=1}^{l} \chi_\lambda(s)^{q^i})(\sum_{i=1}^{l}\chi_\mu(s)^{q^i}).  \]
Recall that for any $1\leq i,j\leq l$, the character \(\chi_\lambda^{q^{i}}\chi_\mu^{q^{j}} \)
is non-trivial on $\mathfrak{t}_\lambda(\mathbb{F}_q)$. Therefore, \[\sum_{s\in \mathfrak{t}_\lambda(\mathbb{F}_q)-\{0\}} \chi_\lambda^{q^{i}}(s)\chi_\mu^{q^{j}}(s)=-1.  \]
We deduce that \begin{equation}\label{Jnilp} J^{\ggg}_{nilp}(\varphi)=l\frac{q-1}{q^l-1} .  \end{equation}

Finally, we conclude from \eqref{Jnilp} and \eqref{10.8} that 
 \begin{equation}J^{G}(e_\rho)  = \begin{cases} 1, \quad \text{$\rho_\lambda$   is contragredient to $\rho_\mu$ };
\\  0, \quad \text{$\rho_\lambda$   is not contragredient to $\rho_\mu$}  .
  \end{cases}  \end{equation}
As we have explained in the beginning of the proof, this suffices to prove the theorem. 
\end{proof}

\end{document}